\documentclass[11pt]{article}
\usepackage[T1]{fontenc}

\usepackage[a4paper,margin=1in]{geometry}
\usepackage{indentfirst}

\usepackage{amsmath,amsthm,amsfonts,amssymb,bm,wasysym}
\usepackage[usenames]{color}
\usepackage{hyperref}
\usepackage[shortlabels]{enumitem}
\usepackage{graphicx}
\usepackage[singlelinecheck=false]{caption}

\numberwithin{equation}{section}

\theoremstyle{definition}
\newtheorem{teo}{Theorem}[section]
\newtheorem{ateo}{Theorem}

\newtheorem{lema}{Lemma}[section]
\newtheorem{prop}{Proposition}[section]
\newtheorem{deft}{Definition}
\newtheorem{cor}{Corollary}[section]

\DeclareMathOperator\PSL{PSL}

\DeclareMathOperator\Cov{Cov}
\DeclareMathOperator\Corr{Corr}
\DeclareMathOperator\Pol{Pol}
\DeclareMathOperator\Area{Area}
\DeclareMathOperator\Leb{Leb}
\DeclareMathOperator\diff{d}
\DeclareMathOperator\Jac{Jac}
\DeclareMathOperator\Diag{Diag}

\newcommand{\A}{\mathbb{A}}
\newcommand{\C}{\mathbb{C}}
\newcommand{\D}{\mathbb{D}}

\newcommand{\s}{\mathbb{S}}
\newcommand{\Z}{\mathbb{Z}}

\allowdisplaybreaks

\title{Continuity of matings of Kleinian groups and polynomials}
\author{Miguel Ratis Laude}

\begin{document}

\maketitle

\begin{abstract}
    In recent years, the study of holomorphic correspondences as dynamical systems that can display behaviors of both rational maps and Kleinian groups has gained a good amount of attention. This phenomenon is related to the Sullivan dictionary, a list of parallels between the theories of these two systems. We build upon a surgical construction of such matings, due to Bullett and Harvey, increasing the degree of maps we consider, and proving regularity properties of the mating map on parameter spaces: namely, analyticity on the interior of its domain of definition, and continuity under quasiconformal rigidity on the boundary.
\end{abstract}

\section{Introduction}

The theory of rational maps and (finitely generated) Kleinian groups have strong parallels, observed since Fatou and more formally exposed by Sullivan in \cite{sullivan-1985}. In \cite{bullett-penrose-1994}, Bullett and Penrose showed that \textit{correspondences}, a kind of multivalued map, could display dynamical behavior similar to both kinds of objects at the same time. Specifically, they managed to show that certain correspondences behaved similarly to quadratic polynomials on a certain limit set, and to representations in $\PSL(2, \C)$ of the modular group $\PSL(2, \Z)$ on the complement of that set. In \cite{bullett-harvey-2000} and \cite{bullett-lomonaco-2019, bullett-lomonaco-2022, bullett-lomonaco-2024}, this has been more concretely proven: the first paper shows that, for any quadratic polynomial with connected Julia set, and any discrete faithful representation of the modular group with connected regular set, there exists some holomorphic correspondence that is the mating between them; and in the other papers, matings of the actual modular group with a certain family of rational maps displaying a persistent parabolic fixed point have been shown to exist. We remark that matings with other groups have also been investigated in \cite{mj-mukherjee-2023, mj-mukherjee-2024}, with certain restrictions on the polynomials, and on a different vain antiholomorphic matings have been studied in \cite{lee-lyubich-makarov-mukherjee-2018, lee-lyubich-makarov-mukherjee-2022}.

In this article, we investigate the \textit{Bullett-Harvey surgery}, the process described in \cite{bullett-harvey-2000}. We first describe an extension of it to general degree, and then apply Douady-Hubbard techniques to understand how the resulting correspondence behaves as we allow the polynomial and the group to vary.

The modular group, which is isomorphic (as a group) to $C_2\ast C_3$ - the free product between the cyclic groups of orders $2$ and $3$, was used as a basis for matings with quadratic polynomials. To increase the degree, we will need to deal with the Hecke groups $H_{d+1} \subset \PSL(2, \C)$. These are isomorphic to $C_2\ast C_{d+1}$ as groups, and in particular $H_3$ is the modular group. Let $\mathcal{D}_{d+1}$ be the space of all discrete representations of $H_{d+1}$ modulo conformal conjugacies. We will speak more on the structure of this space later, but we remark that the interior of $\mathcal{D}_{d+1}$ is comprised of (the conformal conjugacy class of) the representations that display a connected regular set. Our first result is then a generalization of the Bullett-Harvey surgery.

\begin{ateo}\label{main.surgery}

    Let $f$ be a polynomial of degree $d$ with connected filled Julia set, and $r \in \overset{\circ}{\mathcal{D}}_{d+1}$ be a faithful discrete representation of the Hecke group $H_{d+1}$ with connected regular set. Then there exists a holomorphic correspondence $F$ that is a mating between $f$ and $r(H_{d+1})$.
    
\end{ateo}

The precise notion of mating will be defined soon. The techniques used here are similar to the ones used by Bullett and Harvey. The main difference in the construction lies in certain careful choices that make the proofs of the other results easier. We remark that matings with Hecke groups were already looked at in \cite{bullett-freiberger-2003, bullett-freiberger-2005}, and the more recent \cite{bullett-lomonaco-lyubich-mukherjee-2024}, where different surgical processes are presented.

The construction will yield us a mating $F$ between a polynomial and a representation of $H_{d+1}$, which will be a $d:d$ correspondence of the form $J\circ \Cov_0^q$, where $J$ is an involution, $q$ is a degree $d + 1$ polynomial, and $\Cov_0^q$ is the associated \textit{deleted covering correspondence}, defined by mapping a point to all of the other points with the same image under $q$ (see Section 2.3 for the precise definitions). We will denote by $\mathcal{C}_d$ the space of conformal conjugacy classes of such correspondences. It will be obvious from the definition that any correspondence conformal equivalent to a mating is also a mating between the same polynomial and representation.

If now we take polynomials in an analytic family, say $\mathbf{f} = (f_\lambda)_{\lambda\in \Lambda}$, $\Lambda$ a complex manifold, and denote $M_{\mathbf{f}}$ the set of parameters $\lambda$ for which the filled Julia set $K_{f_\lambda} =: K_\lambda$ is connected, the mating process is shown to behave very well in the interior of $M_\mathbf{f}$.

\begin{ateo}\label{main.analyticity}

    Let $\mathbf{f} = (f_\lambda)_{\lambda\in \Lambda}$ be an analytic family of polynomials of the same degree $d$, $\Lambda$ a complex manifold. The map $M_\mathbf{f}\times \overset{\circ}{\mathcal{D}}_{d+1} \to \mathcal{C}_d$ that sends the pair $(\lambda, r)$ to the class of matings between $f_\lambda$ and $r(H_{d+1})$ is analytic on the set $\overset{\circ}{M}_\mathbf{f}\times \overset{\circ}{\mathcal{D}}$.

\end{ateo}

Although the techniques involved in proving Theorem \ref{main.analyticity} are quite classic, there is a very non-trivial work of identifying the correct space in which to apply them. Specifically, we find in $\mathcal{C}_d$ the submanifold passing along a given mating where we have $J$-stability (in a sense similar to Mañé-Sad-Sullivan \cite{mane-sad-sullivan-1983}). That is the step that allows us to follow similar arguments to Douady-Hubbard \cite{douady-hubbard-1985} and conclude the analyticity of the mating process. From the techniques used to prove Theorem \ref{main.analyticity}, it will also be clear that we can make a stronger claim about analyticity with respect to the group parameter, for any given polynomial:

\begin{cor}\label{cor.continuity_along_reps}

    Fix $f$ any polynomial of degree $d$ with connected filled Julia set. The map that sends $r \in \overset{\circ}{\mathcal{D}}_{d+1}$ to the equivalence class in $\mathcal{C}_d$ of matings between $f$ and $r(H_{d+1})$ is analytic.
    
\end{cor}

\textbf{Remark}: This corollary parallels Theorem C of \cite{mj-mukherjee-2024}, where Mj and Mukherjee show that matings of the groups there considered with the polynomial $z^d$ vary holomorphically with the group representation, and lie in a specific slice of the space of correspondences, creating an analogue of a Bers slice. The result presented here talks of the analytic dependence with any fixed polynomial, but we do not identify precisely the slice within $\mathcal{C}_d$ where such matings belong.

Continuity on the boundary of $M_\mathbf{f}$ is subject to \textit{quasiconformally rigidity} of the parameter. A parameter $\lambda \in \Lambda$ is called quasiconformally rigid if every polynomial quasiconformally conjugate to it is actually conformally conjugate; we denote $\mathcal{T}$ the set of quasiconformally rigid parameters that are not in $\overset{\circ}{M}_\mathbf{f}$.

\begin{ateo}\label{main.continuity}

    Let $\mathbf{f} = (f_\lambda)_{\lambda\in \Lambda}$ be an analytic family of polynomials of the same degree $d$, $\Lambda$ a complex manifold. The map $M_\mathbf{f}\times \overset{\circ}{\mathcal{D}}_{d+1} \to \mathcal{C}_d$ that sends the pair $(\lambda, r)$ to the mating between $f_\lambda$ and $r(H_{d+1})$ is continuous on the set $(\overset{\circ}{M}_\mathbf{f}\cup \mathcal{T})\times \overset{\circ}{\mathcal{D}}$.
    
\end{ateo}

As an application, we can look at the unicritical polynomials $\{ f_c(z) = z^d + c \}_{c\in \C}$. Since every point in the boundary of its connectedness locus $M_d$ is quasiconformally rigid, we get the following corollary:

\begin{cor}\label{cor.unicritical_continuity}

    In the case of the unicritical families $\{ f_c(z) = z^d + c \}_{c\in \C}$, the mating map is continuous on the whole product $M_d\times \overset{\circ}{\mathcal{D}}_{d + 1}$.
    
\end{cor}

Another family that displays this property is the family of degree $3$ polynomials $\mathbf{f} = \{ g_a(z) = z^3 + az^2 + z \}_{a\in \C}$ (see Proposition \ref{prop.boundary_rigidity}). The parameter space of this family splits into four pieces: two baby Mandelbrot copies (see \cite{lomonaco-2014}) and two capture regions (see \cite{nakane-2005}).

\begin{cor}\label{cor.persistent_parabolic_family}

    In the case of the persistent parabolic degree $3$ family $\mathbf{f} = \{ g_a(z) = z^3 + az^2 + z \}_{a\in \C}$, the mating map is continuous on the whole product $M_\mathbf{f}\times \overset{\circ}{\mathcal{D}}_{d + 1}$.
    
\end{cor}

This paper is structured in the following way: Section 2 will contain basic definitions and a quick overview of polynomials, representations of Hecke groups, and correspondences; Section 3 is dedicated to the proof of Theorem \ref{main.surgery} and some observations around the types of the correspondences that are obtained as matings; Section 4 will focus on the proofs of Theorems \ref{main.analyticity} and \ref{main.continuity}, and is divided into further subsections. The first subsection of Section 4 will show continuity of the mating map using classical Douady-Hubbard arguments; there, the choices made in the previous section will streamline the proofs. The second subsection tackles the analyticity result with a McMullen-Sullivan flavor; in essence, matings for polynomials within stable components determine quasiconformal deformations of correspondences, which can be encoded by analytic families of Beltrami coefficients over certain domains. Finally, in the third subsection, we prove Theorem \ref{main.continuity}, and introduce examples of families for which the results can be applied.

The author would like to acknowledge Shaun Bullett for several fruitful discussions, and in particular for suggestions regarding the presentation of Lemma \ref{lema.group_structure_stability}.

\section{Preliminaries}

We begin this section by reviewing the theories around polynomials and representations of Hecke groups, with an emphasis in understanding their parameter spaces. We then introduce the relevant object of this text: correspondences. As the surgical process of mating begins with a topological object, and as we will need a notion of convergence, it will be useful to define topological correspondences first. A more detailed study of their basic topological properties can be found in \cite{mcgehee-1992} and \cite{bullett-penrose-2001}.

\subsection{Families of polynomials}

The first ''ingredients'' of the surgery process are polynomials. Their dynamics have been widely studied, and we review here only the necessary definitions and results. Recall that the \textit{filled Julia set} of a polynomial $f$ is the set $K_f \subset \C$ of points with bounded orbit (equivalently, the complement of the basin of attraction of infinity); the \textit{Julia set} $J_f$ of $f$ coincides with the boundary $\partial K_f$.

\begin{deft}

    A family of polynomials $\textbf{f} = (f_\lambda)_{\lambda\in \Lambda}$, $\Lambda$ some complex manifold, is \textit{analytic} if the coefficients of $f_\lambda$ are analytic functions of $\lambda$ --- equivalently, if the map $\textbf{f}(\lambda, z) = f_\lambda(z)$ is holomorphic on $\Lambda\times \overline{\C}$. The \textit{connectedness locus} of $\textbf{f}$ is the set
    \[ M_\textbf{f} := \{ \lambda \in \Lambda \ | \ K_{f_\lambda} \text{ is connected} \}. \]
    
\end{deft}

\begin{figure}
    \centering
    \includegraphics[width=0.3\linewidth]{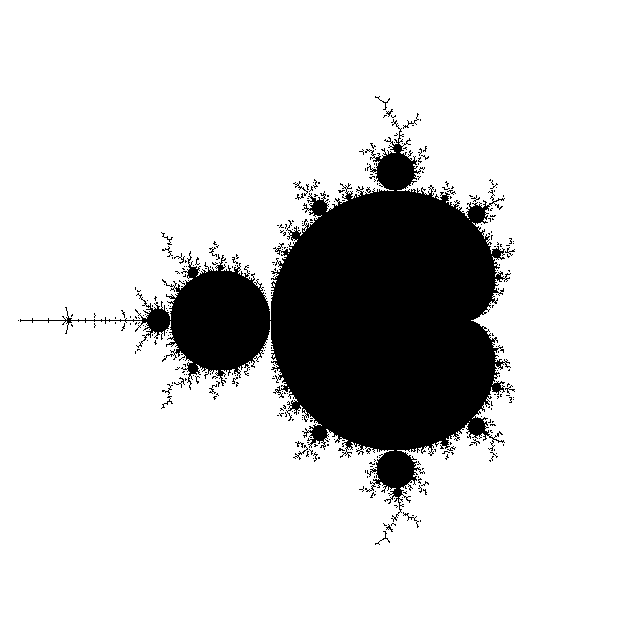}
    \includegraphics[width=0.3\linewidth]{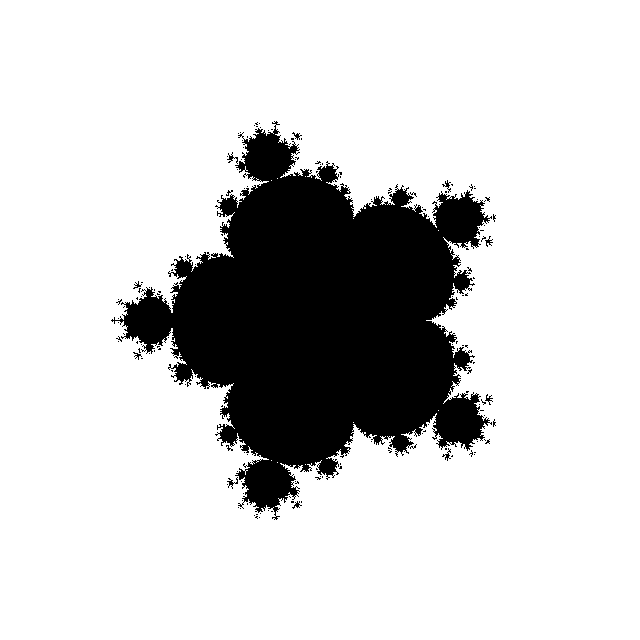}
    \includegraphics[width=0.3\linewidth]{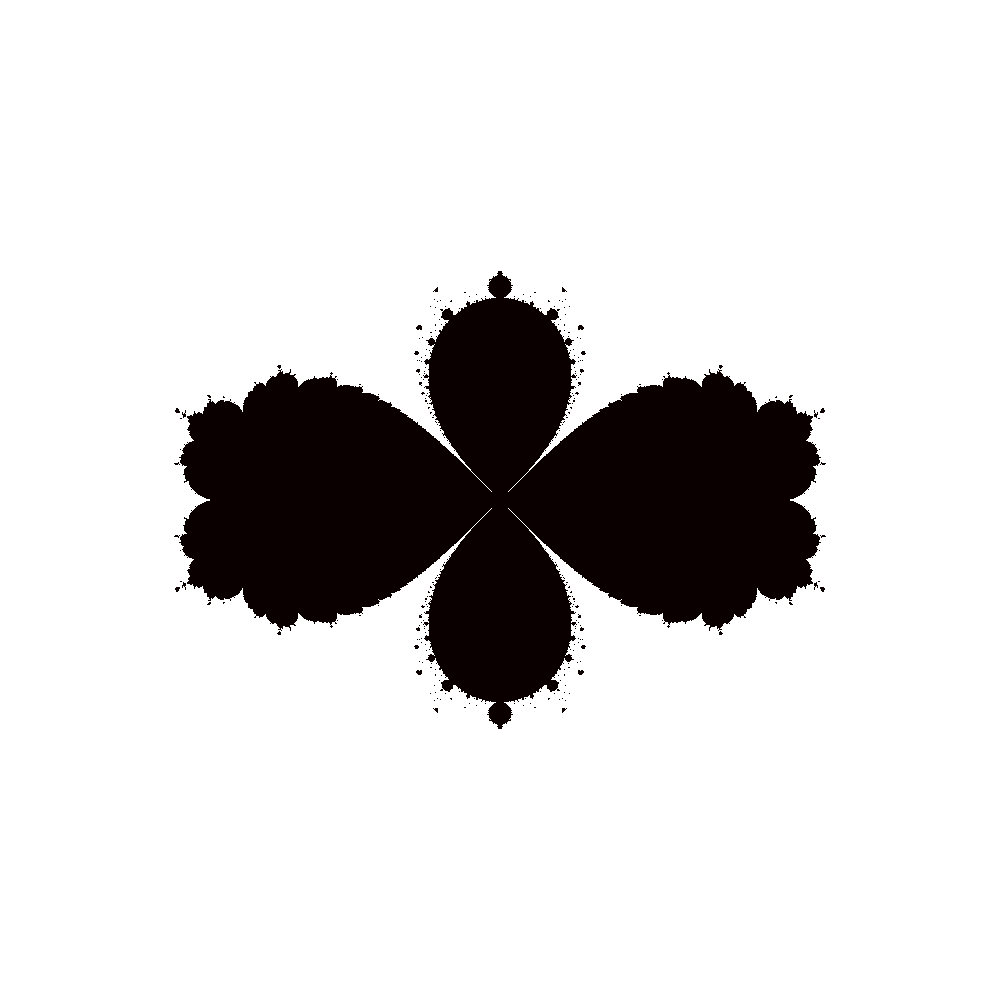}
    \caption{Three examples of connectedness loci of families of polynomials. From left to right: the Mandelbrot set, connectedness locus of the quadratic family $f_c(z) = z^2 + c$; the Multibrot set of degree $5$, connectedness locus of the quintic family $f_c(z) = z^5 + c$; the degree $3$ parabolic butterfly, connectedness locus of the family $g_a(z) = z^3 + az^2 + z$.}
    \label{fig:connectedness_loci}
\end{figure}

Recall that $K_f$ is connected if and only if all (finite) critical points of $f$ belong to $K_f$. From the classical Mañé-Sad-Sullivan theory \cite{mane-sad-sullivan-1983} (complemented by \cite{sullivan-thurston-1986}), any analytic family of polynomials induces a partition $\Lambda = \mathcal{R}_\textbf{f}\cup \mathcal{S}_\textbf{f}$ into disjoint sets, with the following properties:
\begin{itemize}
    \item $\mathcal{R}_\textbf{f}$ is an open dense subset of $\Lambda$;
    \item there is a holomorphic motion of neighborhoods of Julia sets on each connected component of $\mathcal{R}_\textbf{f}$ conjugating the actions of the polynomials on those neighborhoods;
    \item $M_\textbf{f}\cap \mathcal{R}_\textbf{f} = \overset{\circ}{M}_\textbf{f}$;
    \item $\mathcal{S}_f$ is the closure of the set of parameters presenting non-persistent indifferent periodic points.
\end{itemize}
Recall that, given $\lambda_0 \in \Lambda$ and $z_0$ a periodic point of $f_{\lambda_0}$ of period $k$, $z_0$ is said to be \textit{indifferent} if $(f_{\lambda_0}^k)'(z_0) \in \s^1$. We will say then that $z_0$ is a \textit{persistent} indifferent periodic point for $f_{\lambda_0}$ if there is a continuous function $z(\lambda)$ defined in a neighborhood of $\lambda_0$ such that: $z(\lambda)$ is a periodic point of period $k$ for $f_\lambda$; $z(\lambda_0) = z_0$; and
\[ (f_\lambda^k)'(z(\lambda)) = (f_{\lambda_0}^k)'(z_0). \]
Otherwise, we say $z_0$ is a \textit{non-persistent} indifferent periodic point.

From these properties of the Mañé-Sad-Sullivan partition, we see that the interior of the connectedness locus is comprised of \textit{J-stable} components: open sets of parameters over which the actions of the respective maps on their Julia sets are all conjugate to each other.

The theory of polynomials heavily depends on a more general notion: that of \textit{polynomial-like} maps. A polynomial-like map is a holomorphic map $f: U' \to U$ defined on an open subset $U' \subset \C$, satisfying the following properties:
\begin{itemize}
    \item $U', U$ are topological disks;
    \item $U' \Subset U$ --- i.e. $\overline{U}' \subset U$;
    \item $f: U' \to U$ is proper, of some degree $d$.
\end{itemize}
The same definitions and results above for analytic families of polynomials translate to \textit{analytic} families of polynomial-like maps (see \cite{douady-hubbard-1985} for the precise definition). In particular, a similar Mañé-Sad-Sullivan decomposition also exists in this context.

\subsection{Representations of Hecke groups}

Hecke groups are generalizations of the modular group $\PSL(2, \Z)$. While the modular group is generated by the elements
\[ \sigma(z) = \frac{-1}{z} \ \text{ and } \ \rho(z) = -\frac{z + 1}{z} \]
which are rotations on $\overline{\C}$ of angles $\pi$ and $2\pi/3$, respectively, the Hecke group $H_{d+1}$ replaces the order $3$ element by the order $d + 1$ element
\[ \rho(z) = -\frac{2\cos(\pi/(d + 1))z + 1}{z} \]
which is a rotation of angle $2\pi/(d + 1)$. If $r: H_{d+1} \to \PSL(2, \C)$ is a \textit{representation} of $H_{d+1}$ --- i.e. a morphism of groups, we denote $\rho_r := r(\rho)$ and $\sigma_r := r(\sigma)$ the generators of $r(H_{d + 1})$. We want to look at non-trivial representations, which is to say that we want the elements $\rho_r$ and $\sigma_r$ to not be the identity nor powers of each other. We can then use the cross-ratio of the fixed points of these generators as a parametrization for the conjugacy classes of such representations. Indeed, if two representations $r_1, r_2$ produce conformally conjugate images, then the conjugating map has to send the fixed points of each generator of $r_1(H_{d+1})$ to the fixed points of the generators of $r_2(H_{d+1})$, and thus the cross-ratios are either preserved or inverted (since one can always swap the fixed points). Conversely, if the cross-ratios associated to the representations $r_1, r_2$ are the same or inverses of each other, then there is a M\"{o}bius map sending the fixed points of generators of $r_1(H_{d+1})$ to the fixed points of generators of $r_2(H_{d+1})$, and since these generators are rotations, this M\"{o}bius map must then induce a conjugacy between the generators (and therefore the groups). Such cross-ratios can span all complex numbers in $\C\setminus \{0, 1\}$: one can simply consider any four distinct points $z_1, z_2, w_1, w_2$, send $\rho$ to rotation by $2\pi/(d + 1)$ about $z_1$ and $z_2$, and $\sigma$ to rotation by $\pi$ about $w_1$ and $w_2$; this is most obviously seen by setting $w_1 = 0, w_2 = \infty, z_2 = 1$, so that $z_1$ is exactly the cross-ratio. Identifying $z_1$ and $z_1^{-1}$ is the same as taking the image under the map $z \mapsto (z + 1/z)/2$, which maps $\C\setminus\{0, 1\}$ to $\C\setminus\{1\}$.

\begin{deft}

    A representation $r: \Gamma \to \PSL(2, \C)$ of a group $\Gamma$ is
    \begin{itemize}
        \item \textit{faithful} if $r$ is a group isomorphism between $\Gamma$ and its image $r(\Gamma)$;
        \item \textit{discrete} if the image $r(\Gamma)$ is a discrete subgroup of $\PSL(2, \C)$.
    \end{itemize}
    In the case of $\Gamma = H_{d+1}$, we define the \textit{discreteness locus} $\mathcal{D}_{d+1}$ of representations as the set of all conformal conjugacy classes of non-trivial discrete representations.
    
\end{deft}

Using the parametrization via cross-ratios, we can understand $\mathcal{D}_{d + 1}$ as a subset of $\C\setminus \{1\}$. It is a closed set, and its interior $\overset{\circ}{\mathcal{D}}_{d+1}$ is comprised of the parameters associated to discrete faithful representations of $H_{d + 1}$ with connected regular set. This interior is in fact a single connected quasiconformal conjugacy class of quasifuchsian representations of $H_{d + 1}$ of the first kind, and is isomorphic to a punctured disk.

\begin{figure}
    \centering
    \includegraphics[scale = 0.7]{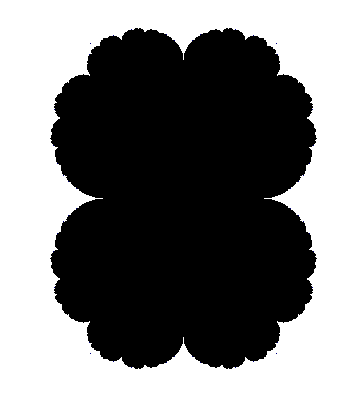}
    \caption{Representation of $\overset{\circ}{\mathcal{D}}_{d+1}$. There is a puncture in the center of the image, and the cusps correspond to certain hyperbolic elements turning into parabolic ones.}
\end{figure}

Given now any non-trivial representation $r$ of $H_{d + 1}$, we can find a unique M\"{o}bius involution $\chi_r$ anti-commuting with $\rho_r$ and $\sigma_r$. Indeed, as is justified in \cite{bullett-harvey-2000}, we can interpret $\rho_r$ and $\sigma_r$ as rotations about certain geodesics on the Poincaré ball. The rotation by $\pi$ about the unique geodesic that intersects both these axes perpendicularly induces on $\overline{\C}$ the action of $\chi_r$. Since $\chi_r$ must exchange the fixed points of $\rho_r$ with each other, we get that $\chi_r\rho_r = \rho_r^{-1}\chi_r$, and similarly $\chi_r\sigma_r = \sigma_r^{-1}\chi_r$ (in this case, this is also equal to $\sigma_r\chi_r$ since $\sigma$ is an involution as well). We will denote $\Gamma(r) := \left< \rho_r, \sigma_r, \chi_r \right>$. We remark that the regular set of $\Gamma(r)$ is the same as the regular set of $r(H_{d + 1})$; we will denote $\Omega(r)$ this common regular set.

\textbf{Remark}: The actual Hecke groups are at the boundary of $\overset{\circ}{\mathcal{D}}_{d + 1}$, since they present parabolic elements. See \cite{bullett-lomonaco-2019, bullett-lomonaco-2022, bullett-lomonaco-2024} for matings with the modular group, and \cite{bullett-lomonaco-lyubich-mukherjee-2024} for matings with Hecke groups.

\subsection{Correspondences}

We finally introduce the objects that will realize our matings: correspondences. The following exposition follows \cite{bullett-penrose-2001}, but we remark that a solid basic theory of correspondences is still lacking, with most results treating only very specific kinds of correspondences.

\begin{deft}

    A \textit{correspondence} $F$ on $\overline{\C}$ is any relation on $\overline{\C}$ --- i.e. any subset of $\overline{\C}\times \overline{\C}$, understood dynamically by setting the \textit{image} of a point $z \in \overline{\C}$ as
    \[ F(z) := \{ w \in \overline{\C} \ | \ (z, w) \in F \}; \]
    see Figure \ref{fig:image_correspondence}. The \textit{inverse} correspondence is the transpose set
    \[ F^{-1} := \{ (z, w) \in \overline{\C}\times \overline{\C} \ | \ (w, z) \in F \}, \]
    and the image under $F^{-1}$ of a point $z \in \overline{\C}$ defines its \textit{pre-image} under $F$. When $F$ is a closed subset of $\overline{\C}\times \overline{\C}$, we say $F$ is a \textit{closed} correspondence. If $F_n \subset \overline{\C}\times \overline{\C}$ is a sequence of closed correspondences, we say that the $F_n$ \textit{converge uniformly} to $F \subset \overline{\C}\times \overline{\C}$ if they converge as compact subsets of $\overline{\C}\times \overline{\C}$ in the Hausdorff topology.
    
\end{deft}

\begin{figure}
    \centering
    \includegraphics[scale = 0.7]{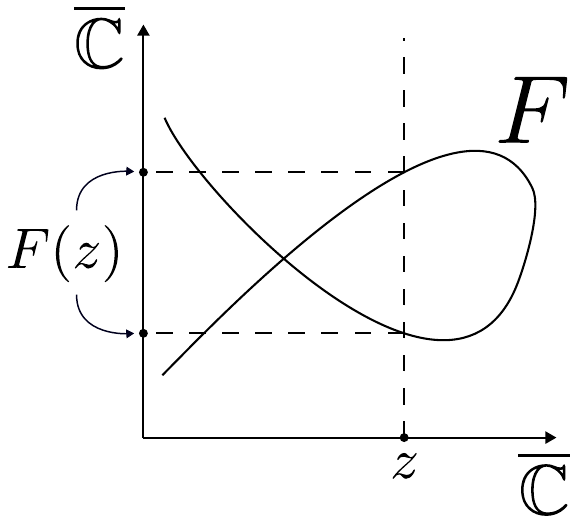}
    \caption{The correspondence $F$ can be understood dynamically a graph of a multivalued function.}
    \label{fig:image_correspondence}
\end{figure}

To understand correspondences as multivalued dynamical systems, we need to be able to iterate them.

\begin{deft}

    If $F, G$ are correspondences on $\overline{\C}$, the composition $G\circ F$ is defined as
    \[ G\circ F := \{ (z, w) \in \overline{\C}\times \overline{\C} \ | \ \exists v \in \overline{\C}, (z, v) \in F, (v, w) \in G \}. \]
    
\end{deft}

It is then quite immediate to show that the composition of closed correspondences is a closed correspondence. We will want to specialize ourselves to the case where correspondences have well defined branches (maybe outside of a finite set of singularities).

\begin{deft}

    Let $F$ be a correspendence on $\overline{\C}$, and let $\pi_i: \overline{\C}\times \overline{\C} \to \overline{\C}$ be the projection onto the $i^{\text{th}}$ coordinate, $i = 1,2$. We then say that $F$ is \textit{open} if the restrictions $\pi_i|_F: F \to \overline{\C}$ are open maps. We say a pair $(z, w) \in F$ is \textit{forward-regular} if $\pi_1|_F$ is a local homeomorphism at that point --- i.e. there is a neighborhood of $(z, w)$ in $F$ where the restriction of $\pi_1|_F$ is a homeomorphism with its image; otherwise, it is called \textit{forward-singular}. A \textit{backward-regular} pair in $F$ is any forward-regular pair of $F^{-1}$; otherwise it is called \textit{backward-singular}. We then say $F$ is a \textit{branched-covering correspondence} if:
    \begin{itemize}
        \item $F$ is closed;
        \item $F$ is open;
        \item images and pre-images of any point in $\overline{\C}$ are discrete (in particular, finite) sets;
        \item the set of points $z \in \overline{\C}$ such that there is some forward-singular pair $(z, w) \in F$ is discrete (in particular, finite);
        \item the set of points $w \in \overline{\C}$ such that there is some backwards-singular pair $(z, w) \in F$ is discrete (in particular, finite).
    \end{itemize}
    
\end{deft}

Around forward-regular pairs, we can define local branches of $F$, by composing $\pi_2\circ \pi_1^{-1}$ on the neighborhood where $\pi_1$ is a homeomorphism; see Figure \ref{fig:correspondence_branches}. As is observed in \cite{bullett-penrose-2001}, branched-covering correspondences have a well defined \textit{bidegree} $(d_1, d_2)$, which means that the image of a generic point has $d_2$ elements, while the pre-image has $d_1$. We say that such a correspondence is \textit{$d_1$ to $d_2$}, and write that as $d_1:d_2$.

\begin{figure}
    \centering
    \includegraphics[scale = 0.7]{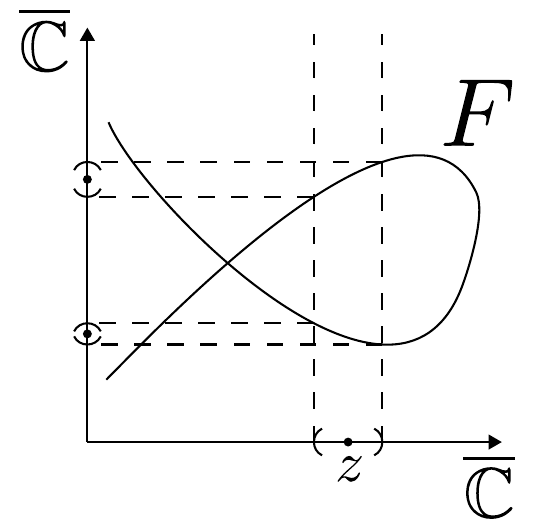}
    \caption{Local branches of a branched-covering correspondence.}
    \label{fig:correspondence_branches}
\end{figure}

\begin{deft}

    If $F$ is a correspondence on $\overline{\C}$ that, as a set, is a complex subvariety of $\overline{\C}\times \overline{\C}$, we call $F$ a \textit{holomorphic correspondence}.
    
\end{deft}

One can verify that holomorphic correspondences are branched-covering. Furthermore, from Chow's theorem, it follows that any holomorphic correspondence on $\overline{\C}$ is actually algebraic, i.e. can be defined as the zero set of a rational map $T$:
\[ F = \{ (z, w) \in \overline{\C}\times \overline{\C} \ | \ T(z, w) = 0 \}. \]
We remark that a rational map $T: \overline{\C}\times \overline{\C} \to \overline{\C}$ is defined to be any algebraic map, when we look at $\overline{\C}$ as the projective complex line --- i.e. in homogeneous coordinates, $T$ is a polynomial, separately homogeneous on the coordinates of each copy of projective space. In other words, $T(z, \cdot)$ and $T(\cdot, w)$ are rational maps for all $z, w \in \overline{\C}$, with degrees (generically) independent on the choices of $z$ and $w$. The bidegree is thus given by the degrees of these rational maps.

Using the fact that holomorphic correspondences have holomorphic local branches, one can show that their composition is also a holomorphic correspondence. Since we are on $\overline{\C}$, we conclude that this composition is also algebraic, again by Chow's theorem. Another consequence of looking at local branches is that, whenever a sequence of holomorphic correspondences $F_n$ converges uniformly to a branched-covering correspondence $F$, we get that $F$ is also holomorphic.

For the purposes of surgery, we need to define a final kind of correspondence.

\begin{deft}

    A branched-covering correspondence $F$ on $\overline{\C}$ is said to be \textit{quasiregular} if any locally defined branch of $F$ (outside its finitely many forward-singularities) is a quasiregular map.
    
\end{deft}

It is then obvious that, if one finds a Beltrami coefficient $\mu$ that is invariant under $F$ --- i.e. invariant under branches of $F$, and take $\phi$ any quasiconformal integrating map --- i.e. $\phi^\ast\mu_0 = \mu$, then the composition $\phi\circ F\circ \phi^{-1}$ will have holomorphic branches, and therefore is a holomorphic/algebraic correspondence.

A trivial example of a correspondence is the diagonal $\Diag := \{ (z, z) \ | \ z \in \overline{\C} \}$. This is the graph of identity, and thus dynamically it behaves as identity. Any rational function defines a correspondence via its graph, and in fact any finite collection of functions (e.g. the generators of a finitely generated Kleinian group) defines a correspondence via the union of their graphs. The simplest examples of correspondences that are not just unions of graphs of actual functions come from deleted covering correspondences.

\begin{deft}
    Let $q$ be a rational map of the sphere. The \textit{covering correspondence} $\Cov^q$ is defined as $\Cov^q := \{ (z, w) \in \overline{\C}\times \overline{\C} \ | \ q(z) = q(w) \}$; notice that $\Diag \subseteq \Cov^q$. The \textit{deleted covering correspondence} $\Cov_0^q$ is defined as
    \[ \Cov_0^q = \left\{ (z, w) \in \C\times \C \ \left| \ \frac{q(z) - q(w)}{z - w} = 0 \right. \right\}, \]
    that is, it is $\Cov^q$ with a copy of $\Diag$ deleted from it.
\end{deft}

The above definition just tells us that $\Cov_0^q$ is mapping $z$ to all points that have the same image as $z$ under $q$, and then deleting one copy of $z$ from this image (i.e. removing one copy of the graph of identity from the curve). Covering correspondences can be very easily classified.

\begin{lema}\label{lema.equivalence_relation_covering}

    If a holomorphic correspondence $F$ is an equivalence relation, then it is a covering correspondence of some rational map.
    
\end{lema}
\begin{proof}

    Since $F$ is holomorphic, the quotient $\overline{\C}/F$ inherits the structure of a Riemann surface. By the Riemann-Hurwitz theorem, this surface has to be a sphere, since the projection map $\pi: \overline{\C} \to \overline{\C}/F$ is holomorphic. Thus, $\pi$ is realized as a rational map $q: \overline{\C} \to \overline{\C}$. It is now obvious that $F = \Cov^q$.
    
\end{proof}

Notice that post-composing $q$ with any M\"{o}bius transformations $\psi$ maintains the covering correspondence the same: $z$ and $w$ are mapped to the same point under $q$ if and only if they are mapped to the same point under $\psi\circ q$. On the other hand, conjugating $\Cov^q$ with a M\"{o}bius transformation $\phi$ results in another deleted covering correspondence, given by $\Cov^{q\circ \phi}$, since $z$ and $w$ have the same image under $q$ if and only if $\phi^{-1}(z)$ and $\phi^{-1}(w)$ have the same image under $q\circ \phi$. Thus, two covering correspondences $\Cov^{q_1}$ and $\Cov^{q_2}$ are conjugate if and only if there are M\"{o}bius transformations $\phi, \psi$ such that $q_2 = \psi\circ q_1\circ \phi$. The same follows for deleted covering correspondences since the identity is conjugate to itself under any map.

\begin{deft}

    We say two rational maps $q_1, q_2: \overline{\C} \to \overline{\C}$ are \textit{conformal covering equivalent} if there are M\"{o}bius maps $\phi, \psi: \overline{\C} \to \overline{\C}$ such that $q_2 = \psi\circ q_1\circ \phi$.
    
\end{deft}

The mated correspondences will arise as compositions $J\circ \Cov_0^q$, where $J$ is an involution and $q$ is a polynomial (see Proposition \ref{prop.classify_matings}). The key feature of these correspondences is the fact that the involution $J$ conjugates it with its inverse; indeed, $(z, w) \in \Cov_0^q\circ J$ if and only if $(J(z), w) \in \Cov_0^q$, if and only if $(w, J(z)) \in \Cov_0^q$ (since deleted covering correspondences are their own inverses), if and only if $(w, z) \in J\circ \Cov_0^q$.

\begin{deft}

    We define $C_d$ the space of all $d:d$ correspondences of the form $J\circ \Cov_0^q$, where $J$ is some involution of $\overline{\C}$ and $q$ is a polynomial of degree $d + 1$. We also let $\mathcal{C}_d$ be the quotient of $C_d$ under the relation of conformal conjugacy.
    
\end{deft}

\begin{prop}

    If the correspondences $J_1\circ \Cov_0^{q_1}$ and $J_2\circ \Cov_0^{q_2}$ are conformally conjugate, then $J_1$ and $J_2$ are conformally conjugate, and $q_1$ and $q_2$ are conformal covering equivalent.
    
\end{prop}
\begin{proof}

    If $\phi$ is the conjugation between the correspondences, then, since $J_1$ is a time-reversing involution for $J_1\circ \Cov_0^{q_1}$ (i.e. it conjugates with the inverse), the map $\phi\circ J_1\circ \phi^{-1}$ is a time-reversing involution for $J_2\circ \Cov_0^{q_2}$, concluding that it must be $J_2$. Indeed, they both are involutions, and they both must have the same fixed points. This means that $\phi$ must also conjugate the deleted covering correspondences, which, as observed before, is equivalent to saying that $q_1$ and $q_2$ are conformal covering equivalent.
    
\end{proof}

Thus, if we want to fix an involution $J$, any correspondence $J'\circ \Cov_0^{q'}$, with $q'$ a polynomial, can be conjugate to one of the form $J\circ \Cov_0^q$, but now $q$ is just a rational map conjugate to a polynomial (equivalently, a rational map with a completely invariant fixed point).

\begin{deft}

    Taking $J$ a M\"{o}bius involution, we define $\Corr_d^J$ the space of all $d:d$ correspondences of the form $J\circ \Cov_0^q$, where $q$ is a rational map conjugate to a polynomial of degree $d+1$.
    
\end{deft}

\section{The Bullett-Harvey Surgery}

In this section, we fix $f$ a degree $d$ polynomial with connected filled Julia set $K_f$ and $r$ a faithful discrete representation of $H_{d + 1}$ with connected regular set. Recall that $K_f$ is the filled Julia set of $f$, that $\rho_r, \sigma_r$ are the generators of $r(H_{d + 1})$, that $\chi_r$ is the unique involution anti-commuting with both these maps, that $\Gamma(r)$ is the group generated by $\rho_r, \sigma_r, \chi_r$, and that $\Omega(r)$ is the common regular set of $r(H_{d + 1})$ and $\Gamma(r)$, as defined in the previous section. We will apply similar ideas from the original construction in \cite{bullett-harvey-2000} to mate these two objects, obtaining a holomorphic correspondence at the end, but a few different choices will be made to facilitate the proofs of Section 4. We first begin by defining the concept of mating we are dealing with.

\begin{deft}

    A holomorphic correspondence $F$ on $\overline{\C}$ is a \textit{mating} between $f$ and $r(H_{d+1})$ if there is a partition $\overline{\C} = \Omega\cup \Lambda$ into completely $F$-invariant, disjoint sets, such that:
    \begin{itemize}
        \item $\Lambda = \Lambda^-\cup \Lambda^+$ is the disjoint union of two compact sets, with $F$ sending $\Lambda^-$ to itself via a $1:2$ branch that is hybrid conjugate to $f$ in a neighborhood of $\Lambda^-$, and $F$ sending $\Lambda^+$ to itself as a $2:1$ correspondence that is hybrid conjugate to $f^{-1}$ in a neighborhood of $\Lambda^+$;

        \item $\Omega$ is an open set in which the action of $F$ is discrete, and the quotients $\Omega/F$ and $\Omega(r)/\Gamma(r)$ are biholomorphic.
    \end{itemize}
    
\end{deft}

We recall that a \textit{hybrid} conjugacy between polynomial-like maps $f: U' \to U, g: V' \to V$ is a quasiconformal conjugacy $\phi$ between (polynomial-like restrictions of) $f$ and $g$, satisfying $\overline{\partial}\phi|_{K_f} \equiv 0$, that is to say, $1$-quasiconformal on the filled Julia set of $f$. Thus, in the first item of the definition, by saying that the action of $F$ on $\Lambda^-$ is hybrid conjugate to that of $f$, we mean that there are neighborhoods $U$ of $K_f$ and $V$ of $\Lambda^-$, and a quasiconformal map $\phi: U \to V$ which is $1$-quasiconformal on $K_f$, conjugating the actions of the polynomial-like restrictions $f: f^{-1}(U) \to U$ and $F: F^{-1}(V)\cap V \to V$. The notion for the action of $F$ on $\Lambda^+$ is analogous.

\begin{figure}
    \centering
    \includegraphics[width=0.9\linewidth]{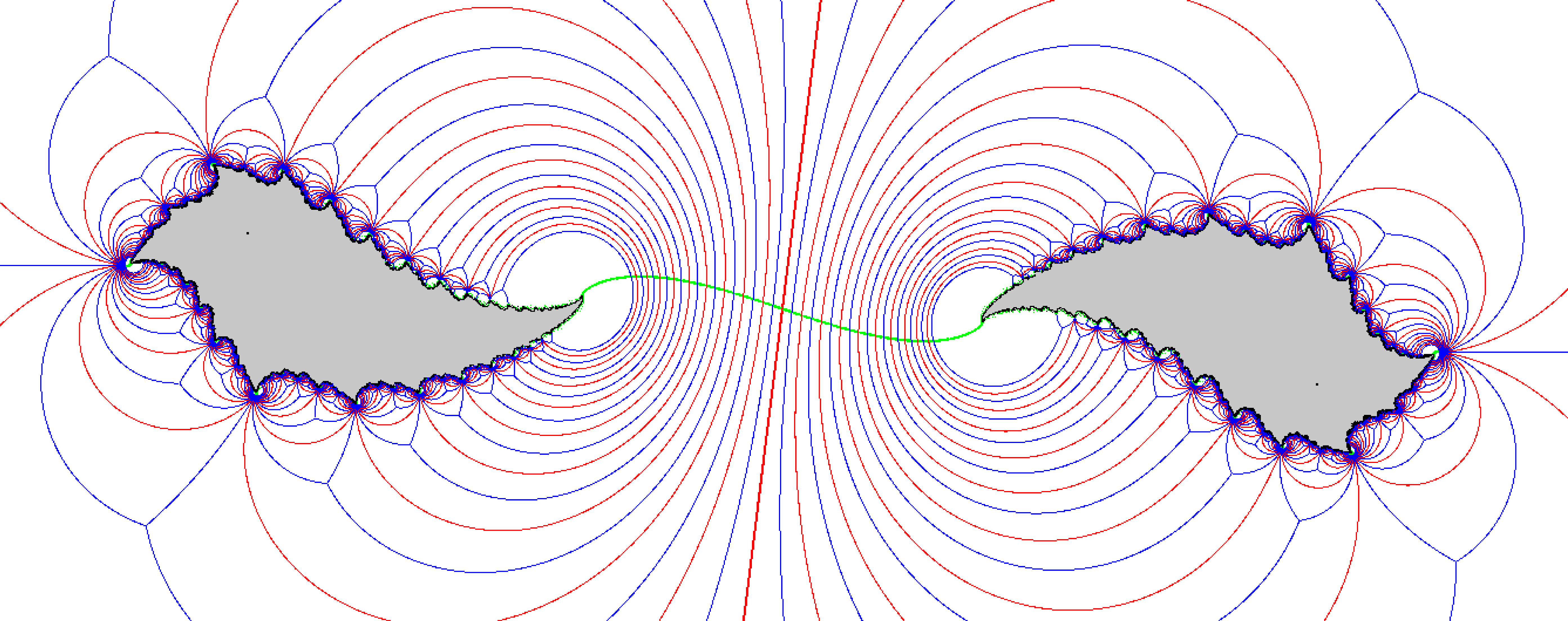}
    \caption{Example of a mating between a quadratic polynomial and a representation of $H_3$. The limit set on the left admits a polynomial-like restriction. The blue and red curves display fundamental domains for the action of $F$ on $\Omega$. This is the correspondence
    \[ F = \left\{ \left( \frac{az + 1}{z + 1} \right)^2 + \left( \frac{az + 1}{z + 1} \right)\left( \frac{aw - 1}{w - 1} \right) + \left( \frac{aw - 1}{w - 1} \right)^2 = 3k \right\} \]
    for $a = 4.53926 + 0.439437 i$ and $k = 0.9 + 0.1 i$.}
    \label{fig:mating_example}
\end{figure}

\textbf{Remark}: Notice that the notion of mating only depends on the polynomial and the group up to conformal conjugacy. Furthermore, if $F$ is a mating of $r(H_{d+1})$ with both $f$ and $\hat{f}$ distinct polynomials with connected filled Julia set, then the first item of the definition will imply that $f$ and $\hat{f}$ have hybrid conjugate polynomial-like restrictions, which by Proposition 6 in \cite{douady-hubbard-1985} means they are conformal equivalent. The same cannot be said of the groups, though; it is easy to see that, given $r \in \overset{\circ}{\mathcal{D}}_{d+1}$, we can define another representation $\hat{r}$ of $H_{d+1}$ by setting $\rho_{\hat{r}} = \rho_r$ and $\sigma_{\hat{r}} = \chi_r\rho_r$. Then $\hat{r} \in \overset{\circ}{\mathcal{D}}_{d+1}$ and $\chi_{\hat{r}} = \chi_r$, and in particular $\Gamma(\hat{r}) = \Gamma(r)$. This means that a mating between $f$ and $r(H_{d+1})$ is also a mating between $f$ and $\hat{r}(H_{d+1})$, and vice-versa, even though $r$ and $\hat{r}$ are not conjugate representations.

To perform surgery between $f$ and $r(H_{d + 1})$, we need to find ''fundamental domains'' for the respective dynamical systems, on the boundaries of which their actions are compatible --- in our case, the domains will be annuli, and the action will be a $d:1$ covering of the inner boundary onto the outer boundary. Such a domain is easy to construct for the map $f$: simply take any polynomial-like restriction $f|_{U'}: U' \to U$ of $f$ and consider $A := U\setminus \overline{U'}$. To make things more explicit, and because this will be useful in the next section, we will choose a precise polynomial-like restriction. Recall that $f$ admits a \textit{B\"{o}ttcher coordinate} $\varphi$, i.e. a map defined in a neighborhood of infinity that conjugates the actions of $f$ and $z^d$. Since we assume that $K_f$ is connected, this map actually extends to an isomorphism $\varphi: \C\setminus K_f \to \C\setminus\overline{\D}$. Thus, fixing any $t > 1$, the sets
\[ U := \varphi^{-1}(\A(1, t^d))\cup K_f, \]
\[ U' := f^{-1}(U) = \varphi^{-1}(\A(1, t))\cup K_f, \]
\[ A := U\setminus \overline{U'} = \varphi^{-1}(\A(t, t^d)), \]
where $\A(t_1, t_2) := \{ z \in \C \ | \ t_1 < |z| < t_2 \}$, satisfy:
\begin{itemize}
    \item $U, U'$ are topological disks, with $U' \Subset U$ --- i.e. $\overline{U'} \subset U$;
    \item the restriction $f|_{U'}: U' \to U$ is a proper holomorphic map of degree $d$;
    \item $A$ is a topological annulus;
    \item if $\partial_i A$ denotes its inner boundary (the one it shares in common with the bounded connected component of $\C\setminus \overline{A}$) and $\partial_o A$ its outer boundary (the one shared in common with the unbounded component), then $f|_{\partial_i A}: \partial_i A\to \partial_o A$ is a $d:1$ covering map.
\end{itemize}

\begin{figure}
    \centering
    \includegraphics[scale = 0.6]{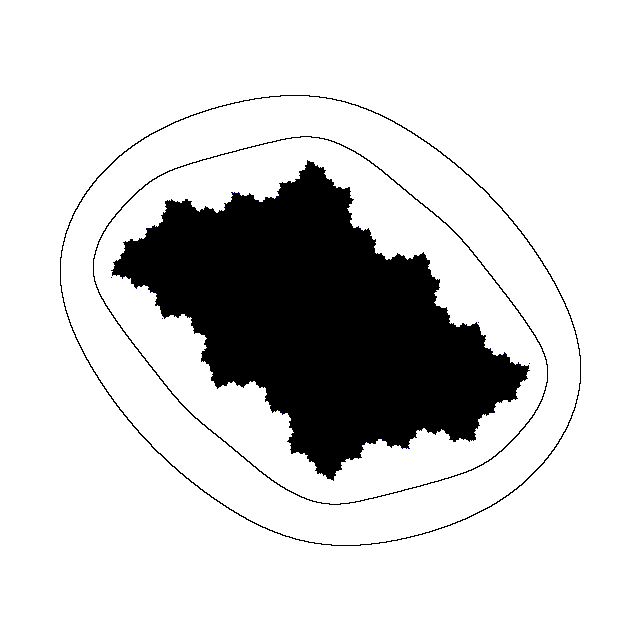}
    \caption{A polynomial-like restriction of the polynomial $f$.}
    \label{fig:polylike_restriction}
\end{figure}

The first two items are exactly what it means for the restriction $f|_{U'}: U' \to U$ to be a polynomial-like map (see Figure \ref{fig:polylike_restriction}). We make observation of an extra property of this choice which will be useful for the construction:
\begin{itemize}
    \item the map
    \[ j: \varphi^{-1}\left( t^de^{it} \right) \mapsto \varphi^{-1}\left( t^de^{-it} \right) \]
    is an orientation-reversing involution of $\partial_o A$ that fixes two points (namely $\varphi^{-1}(t^d)$ and $\varphi^{-1}(-t^d)$, associated with the external rays of arguments $0$ and $\pi$).
\end{itemize}

\textbf{Remark}: We may take $j$ any smooth orientation-reversing involution of $\partial_o A$ for the construction in this section, but canonical choices will again make things easier in the next one.

Turning our focus back to the group representation $r$, let us take $P, P'$ fixed points of $\rho_r$, $Q, Q'$ fixed points of $\sigma_r$, $R, R'$ fixed points of $\chi_r\rho_r$ and $S, S'$ fixed points of $\chi_r\sigma_r$. We also fix smooth curves $\ell, m, n$ connecting $P$ to $R$, $Q$ to $S$ and $R$ to $S$, respectively, in such a way that their projections in the quotient space $\Sigma = \Omega(r)/\Gamma(r)$ do not intersect (also excluding self-intersections). Then the curves $\rho_r(\ell), \sigma_r(m), \chi_r(n)$ connect $P$ to $\rho_rR = \chi_rR$, $Q$ to $\sigma_rS = \chi_rS$, and $\chi_rR$ to $\chi_rS$, respectively. This means that all six curves together bound a domain $\Delta'$ which is a fundamental domain for $r(H_{d+1})$; see Figure \ref{fig:hecke_fundamental_domain}.

\begin{figure}
    \centering
    \includegraphics[scale = 1]{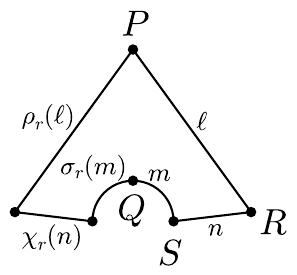}
    \caption{The fundamental domain $\Delta'$ of $r(H_{d+1})$.}
    \label{fig:hecke_fundamental_domain}
\end{figure}

The basis of our ''fundamental annulus'' will be the domain
\[ \Delta := \Delta'\cup \rho_r(\Delta')\cup \dots \cup \rho_r^d(\Delta'). \]
If we quotient $\Delta$ by the action of $\Gamma(r)$, the sides of this domain are identified to each other in several ways, but in particular we have that:
\begin{itemize}
    \item the sides $\chi_r(n), \rho_r(n)$, which are adjacent sides of $\Delta$, are identified under $\chi_r$ with the sides $n, \chi_r\rho_r(n) = \rho_r^d\chi_r(n)$, respectively, which are also adjacent;
    \item the sides $\rho_r^j(n)$, $1 < j \leq d$, are identified under $\rho_r^{2j - 1}\chi_r$ to the sides $\rho_r^{j-1}\chi_r(n)$ --- indeed $\rho_r^{2j - 1}\chi_r = \rho_r^{j-1}\chi_r\rho_r^{-j}$.
\end{itemize}
Identifying only these sides under the appropriate maps turns $\Delta$ into an annulus $B$ --- see Figure \ref{fig:hecke_fundamental_annulus}, whose outer boundary $\partial_o B$ is comprised of the sides $m$ and $\sigma_rm$, while its inner boundary $\partial_i B$ is comprised of the copies of these sides under different powers of $\rho_r$. This means that, when we glue together the actions of all powers of $\rho_r$ on $B$, we obtain a correspondence $g$ whose action along the inner boundary $g|_{\partial_i B}: \partial_i B \to \partial_o B$ is a $d:1$ covering; notice that $g|_{\partial_i B}$ is piecewise smooth. Also, by gluing together the actions of $\sigma_r$ and $\chi_r\sigma_r$ on $\partial_o B$, we again obtain an orientation-reversing involution, which we will also denote $\sigma_r$, fixing the points $Q$ and $S$.

\begin{figure}
    \centering
    \includegraphics[width = 0.95\linewidth]{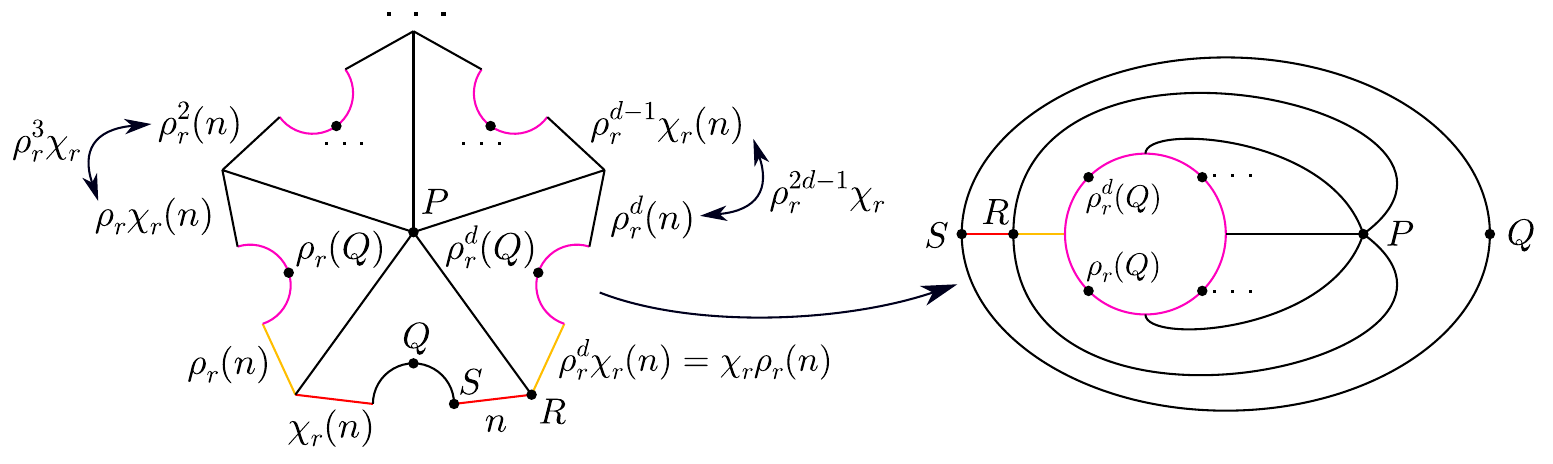}
    \caption{Construction of the ''fundamental annulus'' $B$ for the group representation. The red and yellow sides are identified with each other under the action of $\chi_r$, while the pink segments form the inner boundary; since they are copies of the outer boundary under iterates of $\rho_r$, the map $g$ above acts as a $d:1$ cover.}
    \label{fig:hecke_fundamental_annulus}
\end{figure}

\textbf{Remark}: In the case $d = 2$, all of the above identifications are realized via the element $\chi_r$. This means that $B$ is the image of $\Delta$ on the quotient $\overline{\C}/\left< \chi_r \right>$, the map $g$ is simply the action of the correspondence defined by the two powers of $\rho_r$, restricted to $\partial_i B$, and $\sigma_r$ induces a globally defined involution, restricting to an orientation-reversing involution of $\partial_o B$.

To glue the two annuli together, we will fix a quasiconformal homeomorphism $h: A \to B$ that conjugates $f|_{\partial_i A}: \partial_i A \to \partial_o A, j: \partial_o A \to \partial_o A$ with $g|_{\partial_i B}: \partial_i B \to \partial_o B, \sigma_r: \partial_o B \to \partial_o B$ --- this can be done by starting with any diffeomorphism between $\partial_o A$ and $\partial_o B$ that conjugates $j$ and $\sigma_r$, then lifting to the inner boundaries via $f|_{\partial_i A}$ and $g$, and interpolating inside $A$.

We will now add some extra steps compared to the original construction, so that our work in the next section becomes easier. First, let us embed $B$ into a disk $D$, making $\partial_oB = \partial D$, and preserving the complex structure --- i.e. the structure on $D$ induces the one on $B$. The map $\sigma_r$ then is still a smooth orientation-reversing involution on $\partial D$. We can then take another copy $\tilde{D}$ of $D$ and glue both of them together along the boundary via the map $\sigma_r$ to get a sphere. The obvious extension of $\sigma_r$ to that sphere, which swaps the points of $D$ and $\tilde{D}$, will be a smooth orientation-preserving involution, that restricts to an orientation-reversing involution over the common boundaries between the two disks. For ease of notation, we will simply denote this sphere as $\overline{\C}$, and the embedded annulus still as $B$. Let us also define the sets $D' := D\setminus \overline{B}$, $\tilde{D}' := \tilde{D}\setminus \sigma_r(\overline{B}) = \sigma_r(D')$, and $\tilde{B} := \tilde{D}\setminus \overline{\tilde{D}'} = \sigma_r(B)$. Note that the previously constructed map $h: A \to B$ is still a quasiconformal homeomorphism between these annuli, and thus extends to a quasiconformal homeomorphism $h: U \to D$ such that $h(U') = D'$. We remark that this extension only conjugates dynamics on the boundaries of the fundamental annuli. Let us then define a correspondence $G$ by prescribing the following branches:
\begin{itemize}
    \item $G$ is the conjugated map $h\circ f\circ h^{-1}: D' \to D$ ($d:1$) on $D'$;
    \item $G$ is the conjugated inverse correspondence $\sigma_r\circ h\circ f^{-1}\circ h^{-1}\circ \sigma_r: \tilde{D} \to \tilde{D}'$ ($1:d$) on $\tilde{D}$;
    \item $G$ is the correspondence $\sigma_r\circ \Cov_0^{h\circ f\circ h^{-1}}: D' \to \tilde{D}'$ ($d-1:d-1$) on $D'$ --- i.e. it maps a point $z \in D'$ to the counterparts in $\tilde{D}'$ of all the other points of $D'$ that have the same image under $h\circ f\circ h^{-1}$ as $z$;
    \item $G$ is the correspondence $\sigma_r\circ g: B \to \tilde{B}$ ($d:d$) on $B$.
\end{itemize}

\begin{figure}
    \centering
    \includegraphics[scale = 0.6]{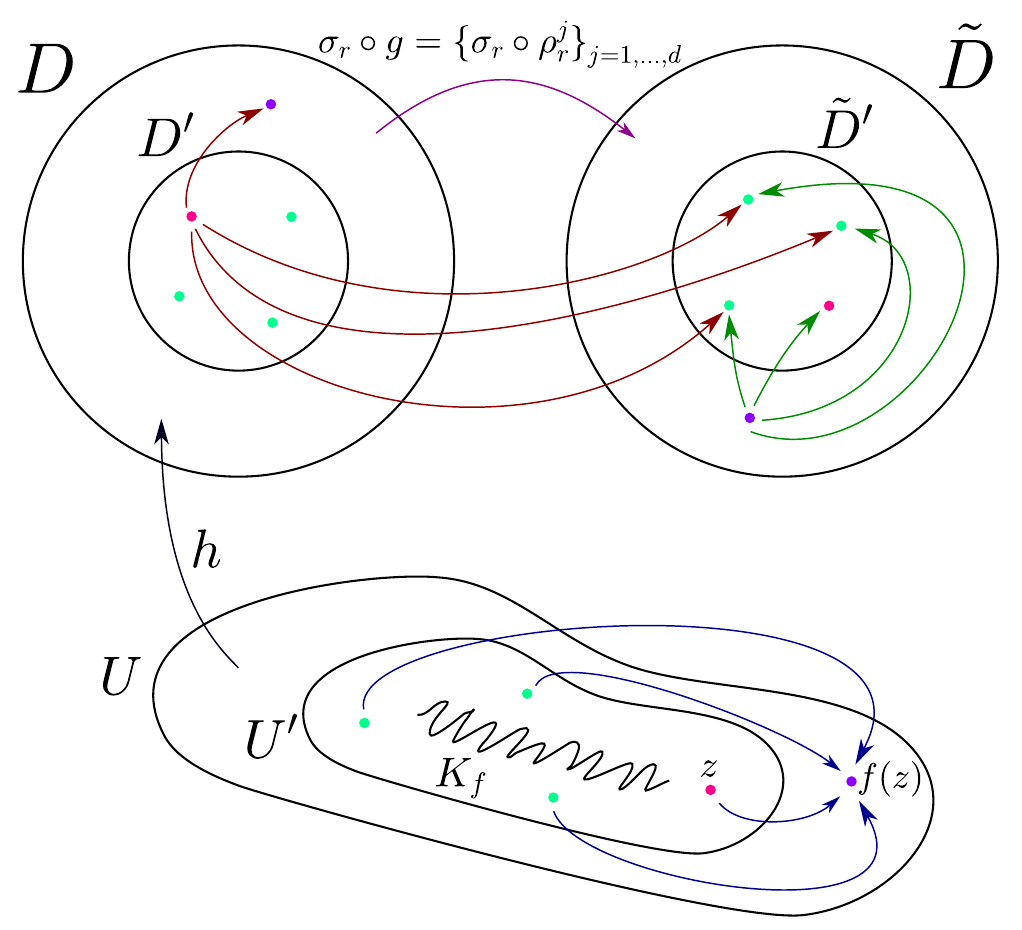}
    \caption{The construction of the topological mating $G$ between the polynomial $f$ and the group $r(H_{d+1})$. The points of same color are associated under either $h$ or $\sigma_r$. The blue arrows show the action of $f$, the red arrows the branches of $G$ on $D'$, the green arrows the branches of $G$ on $\tilde{D}$, and the purple arrow the action of $G: A \to \tilde{A}$.}
    \label{fig:topological_mating_construction}
\end{figure}

It is clear that $G$ is a quasiregular correspondence, and so, to complete the construction, we need only to find a $G$-invariant Beltrami form $\mu$. We shall set
\[ \mu(z) = \begin{cases} \mu_0(z) = 0 & \text{ if } z \in B \iff h^{-1}(z) \in A; \\ (h\circ f^n\circ h^{-1})^\ast\mu(z) & \text{ if } f^n(h^{-1}(z)) \in A; \\ h_\ast\mu_0(z) = (h^{-1})^\ast\mu_0(z) & \text{ if } h^{-1}(z) \in K_f; \\ \sigma^\ast\mu(z) & \text{ if } z \in \tilde{D}. \end{cases} \]
It is clear that $\mu$ is $G$-invariant. If it satisfies $\| \mu \|_\infty < 1$, by what was discussed in the previous section, any integrating map $\phi$ of $\mu$ will be a quasiconformal conjugacy between $G$ and an actual holomorphic correspondence $F$. Notice that $\phi$ also conjugates $\sigma_r$ with a conformal involution $J$, since $\sigma_r$ is quasiconformal and also leaves $\mu$ invariant.

\begin{proof}[Proof of Theorem \ref{main.surgery}]

    The first step is to show that the coefficient $\mu$ constructed above satisfies $\| \mu \|_\infty < 1$. Indeed, that is obvious on the annulus $B$. Since $h$ is $K$-quasiconformal for some $K \geq 1$, and $f$ is holomorphic, we get
    \[ |\mu(z)| \leq \begin{cases} (K - 1)/(K + 1) & \text{ if } h^{-1}(z) \in K_f; \\ (K^2 - 1)/(K^2 + 1) & \text{ if } f^n(h^{-1}(z)) \in A. \end{cases} \]
    Finally, $\sigma_r$ is smooth over the whole sphere, and in particular quasiconformal, concluding that $\| \mu \|_\infty < 1$ as desired.
    
    Now, letting $\phi$ be some integrating map of $\mu$ and $F := \phi\circ G\circ \phi^{-1}$ be the straightening of the topological mating $G$, we see that the restriction $F|_{\phi(D')} = \phi\circ h\circ f\circ h^{-1}\circ \phi^{-1}: \phi(D') \to \phi(D)$ is a polynomial-like restriction of $F$, quasiconformally conjugate to $f: U' \to U$ under the map $\phi\circ h$. By the way we defined $\mu_0$, we also have
    \[ (\phi\circ h|_{K_f})^\ast\mu_0 = (h|_{K_f})^\ast(\phi|_{h(K_f)})^\ast\mu_0 = (h|_{K_f})^\ast\mu|_{h(K_f)} = \mu_0, \]
    showing that $\phi\circ h$ is actually a hybrid conjugacy between the two polynomial-like restrictions. That the action of $F$ on $\phi(\tilde{D})$ is hybrid conjugate to that of the inverse of $f$ on $U$ follows immediately from the relationship between $F$ and $J$. We thus must have $\Lambda := \phi\circ h(K_f)\cup J\circ\phi\circ h(K_f)$. Let us set $\Omega := \overline{\C}\setminus \Lambda$. Since $\mu = \mu_0$ on $B$, $\phi$ is conformal on $B$ and therefore the quotients $\Omega/F$ and $B/G$ are biholomorphic. Since the action of $G$ on $B$ is just the action of $\Gamma(r)$ on $B$, we conclude that $\Omega/F$ is biholomorphic to $\Omega(r)/\Gamma(r)$. This then concludes that $F$ is a mating between $f$ and $r(H_{d+1})$.
    
\end{proof}

This construction allows us to also better understand the structure of the mating $F$.

\begin{prop}\label{prop.classify_matings}

    If $F$ is a mating between a polynomial $f$ of degree $d$ and $r(H_{d + 1})$, $r \in \overset{\circ}{\mathcal{D}}_{d + 1}$, as constructed above, then $F$ is conjugate to a correspondence of the form $J\circ \Cov_0^q$ for some involution $J$ and some degree $d + 1$ polynomial $q$.
    
\end{prop}
\begin{proof}

    As observed in the proof of Theorem \ref{main.surgery} above, the straightening $\phi$ of the topological mating $G$ conjugates $G$ with a true mating $F$, and the smooth involution $\sigma_r$ with an analytic involution $J$ on $\overline{\C}$. We argue that the correspondence $J\circ F$ has to be a deleted covering correspondence. Indeed, notice that the action of $(\sigma_r\circ G)\cup \Diag$ on the annulus $B$ is an equivalence relation; since $\phi$ must conjugate $\Diag$ with itself, we get that $(J\circ F)\cup \Diag$ acts as an equivalence relation on $\phi(B)$. By analyticity, this implies that $(J\circ F)\cup \Diag$ is an equivalence relation on the whole sphere $\overline{\C}$. In particular, by Lemma \ref{lema.equivalence_relation_covering}, it is a covering correspondence, defined by some rational map $q$, and we get that $J\circ F = \Cov_0^q$. Notice now that the point $P \in B$ is completely invariant under the action of $\sigma_r\circ G$. This means that, under an appropriate conjugation, $J\circ F$ has a completely invariant point at infinity, which forces $q$ to have a completely invariant point at infinity, i.e. $q$ is a polynomial.
    
\end{proof}

\textbf{Remark}: As observed above, the fact that $P$ is fixed for $\Cov_0^q$ forces it to be a critical point of $q$ (and in this instance, it is actually fully invariant). But the point $R$ is also fixed by $\Cov_0^q$ (since it acts as $\rho_r$, $\rho_r(R) = \chi_r(R)$, and we've identified $R$ with $\chi_r(R)$ when creating the annulus $B$). But the orbit of $R$ still has $d$ distinct points, and therefore it must be a simple critical point of $q$. In particular, $q$ actually cannot be a polynomial without simple critical points either.

\section{Regularity of the mating process}

For this section, let us fix $\mathbf{f} = (f_\lambda)_{\lambda \in \Lambda}$ an analytic family of polynomials, $\Lambda$ a complex manifold. To prove Theorems \ref{main.analyticity} and \ref{main.continuity}, we will begin using classical Douady-Hubbard methods to show the continuity of the process in $J$-stable components; as observed in Section 1, that means the interior of $M_\mathbf{f}$. Analyticity will follow from the philosophy that the quasiconformal conjugacy class of a correspondence that is a mating between a polynomial and a representation of Hecke group is analytically ''parametrized'' by choices of invariant Beltrami coefficients supported on each invariant set $\Lambda, \Omega$. Finally, the continuity to the boundary will come from the compactness of quasiconformal maps, combined with the hypothesis of rigidity. 

\subsection{Continuity along stable components}

Let us recall that $\mathcal{R}_\mathbf{f}\cap M_\mathbf{f} = \overset{\circ}{M}_\mathbf{f}$.

\begin{teo}\label{teo.continuity101}

    The map $M_\mathbf{f}\times \overset{\circ}{\mathcal{D}}_{d + 1} \to \mathcal{C}_d$ sending the pair $(\lambda, r)$ to (the conjugacy class of) the mating between $f_\lambda$ and $r(H_{d+1})$ is continuous on the set $\overset{\circ}{M}_\mathbf{f}\times \overset{\circ}{\mathcal{D}}$.
    
\end{teo}
\begin{proof}

    Fix $\lambda_0 \in \overset{\circ}{M}_\mathbf{f}$ and $r_0 \in \overset{\circ}{\mathcal{D}}_{d + 1}$. We will begin by finding holomorphic motions of the fundamental domains involved in the surgery. For $\lambda \in M_\mathbf{f}$, let
    \[ \varphi_\lambda: \C\setminus K_\lambda \to \C\setminus \overline{\D} \]
    be the B\"{o}ttcher coordinate for $f_\lambda$. Recall that these maps depend analytically on $\lambda$, since this is true for their germs at infinity. We fix $t > 0$ and take
    \[ U_\lambda := \varphi_\lambda^{-1}(\A(1, t^d))\cup K_\lambda, \]
    \[ U_\lambda' := f_\lambda^{-1}(U_\lambda) = \varphi_\lambda^{-1}(\A(1, t))\cup K_\lambda, \]
    \[ A_\lambda := U_\lambda\setminus \overline{U_\lambda'} = \varphi_\lambda^{-1}(\A(t, t^d)). \]
    Notice that the set $A_\lambda$ is actually well defined for every $\lambda$ in a neighborhood $\mathcal{U}$ of $M_\mathbf{f}$, since the domain of definition of $\varphi_\lambda$ will still include the annulus $\A(t, t^d)$. For ease of notation, let $A_{\lambda_0} =: A_0$. We thus have the following holomorphic motion of $A_0$ along $\mathcal{U}$, based at $\lambda_0$:
    \begin{align*}
        \psi: A_0\times \mathcal{U} & \to \mathbb{C} \\
        (z, \lambda) & \mapsto \varphi_\lambda^{-1}\circ \varphi_{\lambda_0}(z).
    \end{align*}
    Notice that $\psi$ maps the fundamental annulus $A_0$ to the fundamental annulus $A_\lambda$ for every $\lambda \in \mathcal{U}$, conjugating the action of $f_{\lambda_0}|_{\partial_i A_0}: \partial_i A_0 \to \partial_o A_0$ to that of $f_\lambda|_{\partial_i A_\lambda}: \partial_i A_\lambda \to \partial_o A_\lambda$, and the action of $j_0: \partial_o A_0 \to \partial_o A_0$ to that of $j_\lambda: \partial_o A_\lambda \to \partial_o A_\lambda$ (defined in the same way).

    For a group representation $r \in \overset{\circ}{\mathcal{D}}$, let again $\rho_r, \sigma_r$ be the generators of orders $d + 1$ and $2$ of $r(H_{d+1})$, respectively, and $\chi_r$ be the anti-commuting involution. For $r_0$, we also set $\rho_{r_0} =: \rho_0, \sigma_{r_0} =: \sigma_0, \chi_{r_0} =: \chi_0$. If we let $P_0, Q_0, R_0, S_0$ be choices of fixed points of $\rho_0, \sigma_0, \chi_0\rho_0, \chi_0\sigma_0$, then there is a neighborhood $\mathcal{V} \subseteq \overset{\circ}{\mathcal{D}}_{d + 1}$ of $r_0$ in which they vary holomorphically, i.e. we can find holomorphic functions $P,Q,R,S: \mathcal{V} \to \overline{\C}$ such that $P(r) =: P_r, Q(r) =: Q_r, R(r) =: R_r, S(r) =: S_r$ are fixed points of $\rho_r, \sigma_r, \chi_r\rho_r, \chi_r\sigma_r$, respectively. These maps define then a holomorphic motion of the finite set $\{P_0, Q_0, R_0, S_0\}$ along $\mathcal{V}$, based at $r_0$. By Slodkowski's theorem \cite{slodkowski-1991}, we can extend it to a holomorphic motion $\eta_1: \overline{\C}\times \mathcal{V} \to \overline{\C}$ of the whole sphere along $\mathcal{V}$, based at $r_0$.

    Fix now $\ell_0, m_0, n_0$ arcs connecting $P_0$ to $R_0$, $Q_0$ to $S_0$, and $R_0$ to $S_0$, respectively, in such a way that they do not intersect in the quotient $\Omega(r_0)/\Gamma(r_0)$. By reducing the neighborhood $\mathcal{V}$ if necessary, the curves $\eta_1(r, \ell_0), \eta_1(r, m_0), \eta_1(r, n_0)$ are arcs connecting $P_r$ to $R_r$, $Q_r$ to $S_r$, and $R_r$ to $S_r$, respectively, in such a way that they do not intersect in the quotient $\Omega(r)/\Gamma(r)$. If we set $\Delta_0'$ as the domain bounded by the curves $\ell_0, m_0, n_0, \rho_0(\ell_0), \sigma_0(m_0), \chi_0(n_0)$, we can then define a holomorphic motion of $\partial\Delta_0'$ along $\mathcal{V}$, based at $r_0$, in the following way:
    \[ \eta_2(r, z) := \begin{cases} \eta_1(r, z) & \text{ if } z \in \ell_0\cup m_0\cup n_0; \\ \rho_r(\eta_1(r, \rho_0^{-1}(z))) & \text{ if } z \in \rho_0(\ell_0); \\ \sigma_r(\eta_1(r, \sigma_0^{-1}(z))) & \text{ if } z \in \sigma_0(m_0); \\ \chi_r(\eta_1(r, \chi_0^{-1}(z))) & \text{ if } z \in \chi_0(n_0). \end{cases} \]
    Clearly $\eta_2(r, \cdot)$ sends $\partial\Delta_0'$ to some Jordan curve, conjugating the actions of $\Gamma(r_0)$ and $\Gamma(r)$. It again extends by Slodkowski to a holomorphic motion of the whole sphere along $\mathcal{V}$, based at $r_0$, which we restrict to a holomorphic motion of $\Delta_0'$, still denoted $\eta_2$. Let us call $\Delta_r'$ the image of $\Delta_0'$ under $\eta_2(r, \cdot)$; notice that $\Delta_r'$ is a fundamental domain for the action of $\Gamma(r)$.

    Finally, we will define a holomorphic motion of the domain $\Delta_0 := \Delta_0'\cup \rho_0(\Delta_0')\cup \dots\cup \rho_0^d(\Delta_0')$ along $\mathcal{V}$, based at $r_0$, by setting
    \[ \eta(r, z) := \begin{cases} \eta_2(r, z) & \text{ if } z \in \Delta_0'; \\ \rho_r(\eta_2(r, \rho_0^{-1}(z))) & \text{ if } z \in \rho_0(\Delta_0'); \\ \rho_r^{-1}(\eta_2(r, \rho_0(z))) & \text{ if } z \in \rho_0^{-1}(\Delta_0'). \end{cases} \]
    If we denote $\Delta_r$ the image of $\Delta_0$ under $\eta_r := \eta(r, \cdot)$, then $\Delta_r = \Delta_r'\cup \rho_r(\Delta_r')\cup\dots \rho_r^d(\Delta_r')$ and it still projects to an annulus $B_r$ under the identifications considered in the previous section. If $\pi: \Delta_0 \to B_0$ is the projection, the family of Beltrami coefficients
    \[ \mu_r := \pi_\ast\eta_r^\ast(\mu_0|_{\Delta_r}) \]
    is well defined (because $\eta_r$ conjugates group actions, so $\eta_r^\ast\mu_0$ is $\Gamma(r_0)$-invariant in $B_0$), and depends analytically on the parameter $r \in \mathcal{V}$.

    We can now realize the Bullett-Harvey surgery described in the previous section to obtain quasiregular correspondences $G_\lambda$ in the following way: we again take $B_0$ embedded in the sphere $\overline{\C} = D\cup \tilde{D}$, with $B_0 \subset D$, $\partial_o B_0 = \partial D$, and $\tilde{D} = \sigma_0(D)$, where $\sigma_0$ is a smooth extension of the action of $\sigma_0|_{\partial_oB_0}: \partial_oB_0 \to \partial_oB_0$; if $h: U_0 \to D$ is the quasiconformal map extending the one between $A_0$ and $B_0$ (conjugating the actions of $f_0: \partial_iA_0 \to \partial_oA_0, j_0: \partial_o A_0 \to \partial_o A_0$ with that of $g_0|_{\partial_i B_0}: \partial_iB_0 \to \partial_oB_0, \sigma_0: \partial_o B_0 \to \partial_o B_0$, where $g_0$ is the correspondence obtained from the action of $\rho_0$ on $B$), then the mating $G_\lambda$ for $\lambda \in \mathcal{U}$ can be defined in the same way as the topological mating was defined in Section 3, but using the involution $j_\lambda: \partial_o A_\lambda \to \partial_o A_\lambda$, and the map $h_\lambda := \psi_\lambda^{-1}\circ h$ --- we need to extend $\psi_\lambda$ to $U_0 := U_{\lambda_0}$, which can again be done via Slodkowski (this extension does not need to be dynamical). The standard construction straightens it into a conformal mating between $f_\lambda$ and $r_0(H_{d+1})$. To get to a mating between $f_\lambda$ and $r(H_{d+1})$ for $r \in \mathcal{V}$, all we must do is change the Beltrami form that is integrated; we thus define
    \[ \mu_{\lambda, r}(z) := \begin{cases} \mu_r(z) := \eta_r^\ast\mu_0(z) & \text{ if } z \in B_0; \\ (h_\lambda\circ f_\lambda^n\circ h_\lambda^{-1})^\ast\mu_r(z) & \text{ if } f_\lambda^n(h_\lambda^{-1}(z)) \in A_\lambda \iff f_0^n(h_0^{-1}(z)) \in A_0; \\ (h_\lambda)_\ast\mu_0(z) & \text{ if } h_\lambda^{-1}(z) \in K_\lambda; \\ \sigma_0^\ast\mu(z) & \text{ if } z \in \tilde{D}. \end{cases} \]
    Since $\mu_r$ is a Beltrami coefficient on $B_0$, $h_\lambda$ and $\sigma_0$ are quasi-conformal, and $f_\lambda$ is holomorphic, we see that $\mu_{\lambda, r}$ is also a Beltrami coefficient, i.e. $\| \mu_{\lambda, r} \|_\infty < 1$. We can then take $\phi_{\lambda, r}$ an integrating map for $\mu_{\lambda, r}$, and the conjugacy
    \[ F_{\lambda, r} := \phi_{\lambda, r}\circ G_\lambda\circ \phi_{\lambda, r}^{-1} \]
    will be a conformal mating between $f_\lambda$ and $r(H_{d+1})$ (by the same reasons as in the proof of Theorem \ref{main.surgery}).

    The correspondences $G_\lambda$ clearly depend continuously on the parameter $\lambda$, so all we need to do is show that the integrating maps $\phi_{\lambda, r}$ can be taken depending continuously on $(\lambda, r) \in \overset{\circ}{M}_\mathbf{f}\times \mathcal{V}$. By a lemma of Ahlfor's \cite{ahlfors-1966}, it is enough to show that, as $(\lambda, r) \xrightarrow{} (\lambda_0, r_0) \in \overset{\circ}{M}_\mathbf{f}\times \mathcal{V}$, one has $\mu_{\lambda, r} \xrightarrow{L^1} \mu_{\lambda_0, r_0}$ --- i.e. that the convergence occurs in $L^1$ norm. For that, let us define the ''truncated'' coefficients
    \[ \mu_{\lambda, r}^{(n)}(z) := \begin{cases} \mu_{\lambda, r}(z) & \text{ if } f_0^k(h_0^{-1}(z)) \in A_0 \text{ for some } 0 \leq k \leq n; \\ 0 & \text{ otherwise}. \end{cases} \]
    It is clear that $\mu_{\lambda, r}$ is the pointwise limit of $\mu_{\lambda, r}^{(n)}$ as $n \xrightarrow{} \infty$, and that each $\mu_{\lambda, r}^{(n)}$ depends $L^1$-continuously on $(\lambda, r)$ --- indeed, the coefficient $\mu_r := \eta_r^\ast\mu_0$ on $B_0$ depends analytically on $r$, and there are finitely many ''pieces'' $(h_\lambda\circ f_\lambda^n\circ h_\lambda^{-1})^\ast\mu_r$ on pre-images of $B_0$ that depend analytically on $\lambda$. It is then enough to show that
    \[ \mu_{\lambda, r}^{(n)} \xrightarrow{L^1} \mu_{\lambda, r} \ \text{ locally uniformly on } \overset{\circ}{M}_\mathbf{f}\times\overset{\circ}{\mathcal{D}}. \]
    From the fact that $h_0$ is quasiconformal, we find that this is equivalent to showing that
    \[ h_0^\ast\mu_{\lambda, r}^{(n)} \xrightarrow{L^1} h_0^\ast\mu_{\lambda, r} \ \text{ locally uniformly on } \overset{\circ}{M}_\mathbf{f}\times\overset{\circ}{\mathcal{D}}. \]
    We shall denote $\hat{\mu}_{\lambda, r} := h_0^\ast\mu_{\lambda, r}$ and $\hat{\mu}_{\lambda, r}^{(n)} := h_0^\ast\mu_{\lambda, r}^{(n)}$.
    
    Since $\lambda_0 \in \mathcal{R}_\mathbf{f}$, we can find a neighborhood $V_0$ of $J_{\lambda_0}$ and a holomorphic motion along $W \subseteq \mathcal{R}_\mathbf{f}$ the connected component containing $\lambda_0$, based at $\lambda_0$,
    \[ \tau: W\times V_0 \mapsto \C \]
    that conjugates $f_{\lambda_0}$ on $V_0$ with $f_\lambda$ on $V_\lambda := \tau(\lambda, V_0)$ for every $\lambda \in W$ (see \cite{mane-sad-sullivan-1983}). If we let $U''_0 := V_0\cup K_{\lambda_0}$ and $U''_\lambda := V_\lambda\cup K_\lambda$, we can find some $k \geq 0$ such that $f_{\lambda_0}^{-k}(A_0) \subset U''_0$; by continuity, we can also find a small neighborhood $\mathcal{U}' \subseteq W$ of $\lambda_0$ such that, whenever $\lambda \in \mathcal{U}'$, $f_\lambda^{-k}(A_\lambda) \subset U''_\lambda$ as well. If we show that
    \[ \Area\left( f_\lambda^{-n}(U''_\lambda)\setminus K_\lambda \right) \xrightarrow{} 0 \ \text{ uniformly on } \lambda \in \mathcal{U}', \]
    we are done, since then
    \[ \int{ \left| \hat{\mu}_{\lambda, r} - \hat{\mu}_{\lambda, r}^{(n)} \right| }\diff\Leb = \int_{f^{-n}(U''_\lambda)\setminus K_\lambda}{ |\hat{\mu}_{\lambda, r}| }\diff\Leb \leq \Area(f^{-n}(U''_\lambda)\setminus K_\lambda)\| \hat{\mu}_{\lambda, r} \|_\infty \]
    and $\| \hat{\mu}_{\lambda, r} \|_\infty < 1$. To conclude, we observe that, because $\tau$ is a holomorphic motion, and maybe after reducing the size of $\mathcal{U}'$, every $\tau_\lambda := \tau(\lambda, \cdot)$ for $\lambda \in \mathcal{U}'$ is quasiconformal with a bound on distortion that does not depend on $\lambda$. This implies that
    \[ \frac{1}{C}\| D\tau_\lambda \|^2 \leq \Jac(\tau_\lambda) \leq C\| D\tau_\lambda \|^2 \]
    for every $\lambda \in \mathcal{U}'$ and some constant $C \geq 1$ not depending on $\lambda$. This in turn implies that
    \[ \frac{1}{C}\int_{f^{-n}(U''_0)\setminus K_0}{ \| D\tau_\lambda \|^2 }\diff\Leb \leq \Area(f^{-n}(U''_\lambda)\setminus K_\lambda) \leq C\int_{f^{-n}(U''_0)\setminus K_0}{ \| D\tau_\lambda \|^2 }\diff\Leb \]
    since $\Area(f^{-n}(U''_\lambda)\setminus K_\lambda) = \int_{f^{-n}(U''_0)\setminus K_{\lambda_0}}{ \Jac(\tau_\lambda) }\diff\Leb$. This is enough to conclude the result since the functions
    \[ \lambda \mapsto \int_{f^{-n}(U''_0)\setminus K_0}{ \| D\tau_\lambda \|^2 }\diff\Leb \]
    form a decreasing sequence of plurisubharmonic functions that converge pointwise to $0$, and therefore converge locally uniformly to $0$.

\end{proof}

\textbf{Remark}: Using the same proof as before, we can actually verify that the map sending $r$ to the mating of $f_\lambda$ and $r(H_{d+1})$ is continuous for any $\lambda \in M_\mathbf{f}$. Indeed, the same steps can be repeated to show that it is enough to prove the convergence $\mu_{\lambda, r}^{(n)} \xrightarrow{L^1} \mu_{\lambda, r}$ locally uniformly on $r$, and that follows immediately from the fact that
\[ \Area(f^{-n}(U''_\lambda)\setminus K_\lambda) \xrightarrow{} 0 \]
for any $\lambda \in M_\mathbf{f}$ independently of $r$ (just not uniformly on $\lambda$).

\subsection{Proof of Theorem \ref{main.analyticity}}

We focus now on showing that the mating process is analytic in $\overset{\circ}{M}_\mathbf{f}\times \overset{\circ}{\mathcal{D}}_{d + 1}$. The idea is to first identify exactly the set where matings of representations of the Hecke group with the family $\mathbf{f}$ lie in $\mathcal{C}_d$. It will actually be easier to work with the more concrete family $\Corr_d^J$, where each element is an actual correspondence instead of an equivalence class, and the time reversing involution is the same for all of them. The proof will be preceded by a series of lemmas. Throughout, let us fix parameters $\lambda_0 \in \overset{\circ}{M}_{\mathbf{f}}$ and $r_0 \in \overset{\circ}{\mathcal{D}}_{d+1}$. We shall denote by $W$ the connected component of $\overset{\circ}{M}_{\mathbf{f}}$ containing $\lambda_0$. Fix also $J$ any involution, and $F_0 \in \Corr_d^J$ a mating between $f_{\lambda_0}$ and $r_0(H_{d+1})$.

\begin{lema}\label{lema.poly_like_stability}

    There exists a neighborhood $\mathcal{W}_1 \subset \Corr_d^J$ of $F_0$ and some open simply connected set $V \subset \overline{\C}$ such that:
    \begin{itemize}
        \item for every $F \in \mathcal{W}_1$, the set $V_F := F^{-1}(V)$ is simply connected, and $F|_{V_F}: V_F \to V$ is a polynomial-like map;
        \item $\partial V$ is $J$-invariant, containing the fixed points of $J$.
    \end{itemize}
    
\end{lema}
\begin{proof}

    Let $G$ be the topological mating between $f_{\lambda_0}$ and $r_0$ constructed in Section 3, and $\phi$ be the straightening map --- i.e. $\phi\circ G\circ \phi^{-1} = F_0$. In that construction, we obtain $\overline{\C}$ as the union of two discs $D$ and $\tilde{D}$, glued together along their boundaries, with $G^{-1}(D) = D'$ satisfying that $G|_{D'}: D' \to D$ is conjugate to a polynomial-like map. Setting then $V := \phi(D)$, we have $V_{F_0} = F_0^{-1}(V) = \phi(D')$ and $F_0|_{V_{F_0}}: V_{F_0} \to V$ is a polynomial-like map. This means that, if $\mathcal{W}_1$ is a small enough neighborhood of $F_0$, all restrictions $F|_{V_F}: V_F = F^{-1}(V) \to V$ for $F \in \mathcal{W}_1$ will also be polynomial-like maps. The observation that $\partial V$ is $J$-invariant, containing the fixed points of $J$, comes directly from observing that $J = \phi\circ \sigma_{r_0}\circ \phi^{-1}$, and $\partial D$ is $\sigma_{r_0}$-invariant, and contains the fixed points of $\sigma_{r_0}$.
    
\end{proof}

As a consequence of how the polynomial-like restrictions are obtained, we observe that the family
\[ \{ F|_{V_F}: V_F \to V \}_{F \in \mathcal{W}_1} \]
is an analytic family of polynomial-like maps. The neighborhood $\mathcal{W}_1$ can be thought of as a neighborhood of matings around $F_0$, although the polynomial-like restrictions don't have to necessarily present connected filled Julia sets. We still need to verify that the correspondences in $\mathcal{W}_1$ act as groups outside the two copies of filled Julia sets, which will come as a direct consequence of the following result. To fix notation, we define the annulus $A_F := V\setminus \overline{V_F}$.

\begin{lema}\label{lema.group_structure_stability}

    Let $D \subset \mathcal{W}_1$ be any embedded disk containing $F_0$. Then there exists a holomorphic motion
    \[ \eta: A_{F_0}\times D \to \overline{\C} \]
    of the annulus $A_{F_0}$ along $D$, based on $F_0$, such that $\eta_F(A_{F_0}) = A_F$, and $\eta_F$ conjugates the $d:d$ actions of $J\circ F_0|_{A_{F_0}}: A_{F_0} \to A_{F_0}$ and $J\circ F|_{A_F}: A_F \to A_F$.
    
\end{lema}
\begin{proof}

    To construct the holomorphic motion, we will set $\eta(z, F) = z$ for every $z \in \partial V$; this determines the holomorphic motion of $\partial V_{F_0}$ by conjugating dynamics:
    \[ \eta(z, F) := F^{-1}(\eta(F_0(z), w)), \]
    where the inverse is taken to be the branch mapping closest to $z$. In other words, the map $\eta_F = \eta(\cdot, F)$ on $\partial V_{F_0}$ is the lift under the pair of degree $d$ covering maps $F_0|_{\partial V_{F_0}}:\partial V_{F_0} \to \partial V, F|_{\partial V_F}:\partial V_F \to V$ of the map $\eta_F$ on $\partial V$. Now notice that, for all $F \in \Corr_d^J$, the deleted covering correspondence $J\circ F =: \Cov_0^{q_F}$ has a completely invariant point, which must depend holomorphically on $F$. Let us denote by $P_F$ this point. Since $P_{F_0} \in A_{F_0}$, being the image of the fixed point $P$ of $\rho_r$ in its construction, we get $P_F \in A_F$ for all $F \in \mathcal{W}_1$, up to shrinking this neighborhood further. We will also keep track of some other points: the image of the fixed point $S$ of $\chi_r\sigma_r$, which stays fixed as we move $F$ since it is one of the two fixed points of $J$, and belongs to $\partial V$; the image of the fixed point $R$ of $\chi_r\rho_r$, so that $R_{F_0} \in A_{F_0}$ is a simple critical point of $q_{F_0}$ (see the remark after Proposition \ref{prop.classify_matings}), which must then move holomorphically with $F$; and finally we choose some point $R'$ in the $\rho$-orbit of $R$, and since these points are discrete we can make $R'$ move holomorphically, up to again shrinking the size of $\mathcal{W}_1$. This then extends $\eta$ from the boundary of the annulus to now three more points, $P_{F_0}, R_{F_0}$, and $R'_{F_0}$ (see step $1$ in Figure \ref{fig:holomorphic_motion_adaptation}).
    
    Going back to the construction of the ''fundamental'' annulus (see Figure \ref{fig:hecke_fundamental_annulus} again), our goal is to extend the holomorphic motion to one of the curves that makes up the boundary of a true fundamental domain of $\Gamma$. Let us then call $\gamma$ the curve that starts at $P$ and and passes trough $R'$. The corresponding curve on $A_{F_0}$ will be denoted $\gamma_{F_0}$. This extension can be done via Slodkowski \cite{slodkowski-1991}, producing curves $\gamma_F$ for each $F \in D$, but we must ensure that the images under $\Cov_0^{q_F}$ of $\gamma_F$ do not intersect each other (other than at the common point $P_F$) --- otherwise, we will not be able to spread the holomorphic motion by the dynamics. Since the pieces of the images of $\gamma_{F_0}$ under $\Cov_0^{q_{F_0}}$ outside of a neighborhood of $P_{F_0}$ are all at some positive distance from one another, the only problems that may arise are exactly in a neighborhood of $P_F$ (we may always reduce the size of $\mathcal{W}_1$ and therefore of $D$ if necessary). Since $P_F$ is a superattracing fixed point of $q_F$ of local degree $d+1$, we may find locally defined B\"{o}ttcher coordinates that vary analytically with $F$. Let us then fix a small radius $\varepsilon > 0$ such that the inverse B\"{o}ttcher coordinates $\psi_F: D(0, \varepsilon) \to \overline{\C}$ are defined for all $F \in \mathcal{W}_1$. Notice that $\psi_F$ conjugates the local action of $\Cov_0^{q_F}$ with that of $\Cov_0^{z^{d + 1}}$; in particular, if $\psi_{F_0}^{-1}(\gamma_{F_0}) = \gamma_{F_0}'$ (this is only defined for a part of $\gamma_{F_0}$ close enough to $P_{F_0}$), then
    \[ \psi_{F_0}^{-1}(\Cov_0^{q_F}(\gamma_{F_0})) = e^{2\pi i/d}\gamma_{F_0}'\cup \dots \cup e^{2\pi i(d - 1)/d}\gamma_{F_0}'. \]
    We then extend the holomorphic motion $\eta$ by setting $\eta(\psi_{F_0}(t), F) := \psi_F(t)$ for all
    \[ t \in \gamma_{F_0}\cup e^{2\pi i/d}\gamma_{F_0}'\cup \dots \cup e^{2\pi i(d - 1)/d}\gamma_{F_0}'. \]
    (which are the begining piece of $\gamma_{F_0}$ and its images under $\Cov_0^{q_{F_0}}$), and allowing any extension to the rest of the curve $\gamma_{F_0}$, by Slodkowski (see steps 2 and 3 in Figure \ref{fig:holomorphic_motion_adaptation}).

    Finally, there is a $2:2$ branch of $\Cov_0^{q_{F_0}}$ that sends the curve $\gamma_{F_0}$ to a tree branching at $R_{F_0}$. Since $R_F$ is still critical for $q_F$, the same happens for $\gamma_F$: there is a $2:2$ branch of $\Cov_0^{q_F}$ sending $\gamma_F$ to a tree branching at $R_F$. Thus, we have successfully extended $\eta$ to the boundaries of the fundamental domains of $\Cov_0^{q_{F_0}}$. We may then apply Slodkowski again to extend $\eta$ to any such fundamental domains, and the spread by the dynamics to conclude the result (see step 4 in Figure \ref{fig:holomorphic_motion_adaptation}).
    
\end{proof}

\begin{figure}
    \centering
    \includegraphics[scale = 0.55]{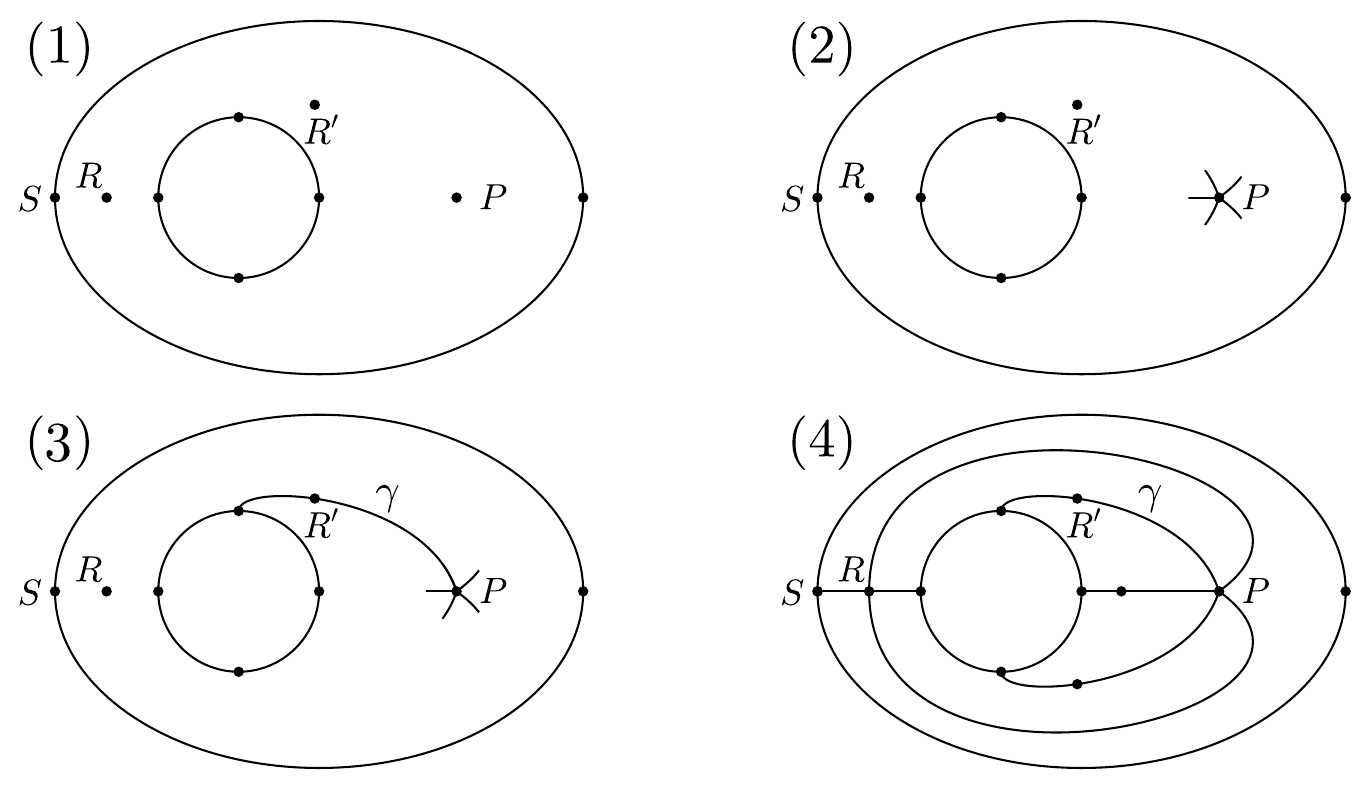}
    \caption{The steps of extending the holomorphic motion in Lemma \ref{lema.group_structure_stability}. Step 1 starts with a holomorphic motion of the boundary of the ''fundamental'' annulus, plus a few dynamically determined points. Step 2 includes small curves around $P$ that are images of each other under powers of $\rho$; they arise from B\"{o}ttcher coordinates. Step 3 extends one of these curves to $\gamma$, whose images will bound fundamental domains for the action of $\rho$. Step 4 has the images of $\gamma$ under powers of $\rho$ drawn, using the fact that $R$ is a simple critical point of the defining rational map to obtain the branching picture.}
    \label{fig:holomorphic_motion_adaptation}
\end{figure}

\textbf{Remark:} Notice that it is necessary that we stay in the family of compositions $J\circ \Cov_0^q$ with $q$ equivalent to a polynomial, since otherwise the fixed point of the order $d + 1$ element would not be present. These can still be matings in a different sense (see the recent \cite{mj-mukherjee-2024}, where Mj and Mukherjee produce correspondences that display an invariant cycle instead of a fixed point).

Since $D$ can be any embedded disk, we see that every correspondence $F \in \mathcal{W}_1$ presents a group structure outside of the filled Julia set of the polynomial-like restriction $F|_{V_F}: V_F \to V$ and its image under $J$. That is to say, we have a standard decomposition into $F$-invariant sets $\overline{\C} = \Lambda_F\cup \Omega_F$, where $\Lambda_F$ is the union of the two aforementioned filled Julia sets, and $\Omega_F/F$ will be biholomorphic to some $\Omega(r)/\Gamma(r)$, $r \in \overset{\circ}{\mathcal{D}}_{d+1}$. We now restrict the neighborhood $\mathcal{W}_1$ to a smaller set where we can hope to find matings with members of the family $\mathbf{f}$. The following is a consequence of the theory developed by McMullen and Sullivan in \cite{sullivan-mcmullen-1998}.

\begin{lema}\label{lema.J_stable_manifold}

    There exists an analytic suborbifold $\mathcal{W}_2 \subseteq \mathcal{W}_1$ containing $F_0$ such that every $F \in \mathcal{W}_2$ has its polynomial-like restriction $F|_{V_F}: V_F \to V$ quasiconformally conjugate to the polynomial-like map $F_0|_{V_{F_0}}: V_{F_0} \to V$.
    
\end{lema}

In other words, the quasiconformality class of the correspondence $F_0$ forms a suborbifold of $\mathcal{W}_1$. Since we are only interested in local results, we may reduce the size of $\mathcal{W}_1$ in order to find a complex manifold $\hat{\mathcal{W}}_2$ and a holomorphic map $\pi: \hat{\mathcal{W}}_2 \to \mathcal{W}_1$ such that $\pi(\hat{\mathcal{W}}_2) = \mathcal{W}_2$. Thus, for each $w \in \hat{\mathcal{W}}_2$, we have associated a correspondence $F_w := \pi(w) \in \mathcal{W}_2$, and the family
\[ \{ F_w|_{V_{F_w}}: V_{F_w} \to V \}_{w\in \hat{\mathcal{W}}_2} \]
is an analytic family of polynomial-like maps. Furthermore, since all of them are quasiconformally conjugate to each other, they all have connected filled Julia set. Lemma \ref{lema.group_structure_stability} thus implies that all of the $F_w$ are matings between some polynomial $f$ quasiconformally conjugate to $f_{\lambda_0}$ and some $r(H_{d+1})$, $r \in \overset{\circ}{\mathcal{D}}_{d+1}$. In general, it is not true that all $\lambda$ close to $\lambda_0$ have $f_\lambda$ quasiconformally conjugate to $f_{\lambda_0}$. Nonetheless, by results in \cite{sullivan-mcmullen-1998}, we know that there is an open dense subset $W'$ of $W$ where every pair $\lambda, \lambda' \in W'$ have that $f_\lambda$ and $f_{\lambda'}$ are quasiconformally conjugate. We may then assume that $\lambda_0 \in W'$, since we know from Theorem \ref{teo.continuity101} that the mating map is continuous, and therefore it is enough to show that it is analytic in an open dense subset of $\overset{\circ}{M}_{\mathbf{f}}\times \overset{\circ}{\mathcal{D}}_{d+1}$.

We now reduce the domain $\hat{\mathcal{W}}_2$ even further to identify the mating with maps from the family $\mathbf{f}$. Let us first fix $w_0 \in \hat{\mathcal{W}}_2$ be such that $\pi(w_0) = F_0$.

\begin{lema}\label{lema.mating_family_manifold}

    There exists an analytic subvariety $\mathcal{W} \subseteq \hat{\mathcal{W}}_2$ containing $w_0$ such that, for all $w \in \mathcal{W}$, the polynomial-like restriction $F_w|_{V_{F_w}}: V_{F_w} \to V$ of the correspondence $F_w = \pi(w)$ is hybrid conjugate to a polynomial $f_\lambda$, for some $\lambda$ close to $\lambda_0$.
    
\end{lema}
\begin{proof}

    The family of polynomial-like maps
    \[ \{ F_w|_{V_{F_w}}: V_{F_w} \to V \}_{w\in \hat{\mathcal{W}}_2} \]
    is analytic. Since all of these maps are quasiconformally conjugate, there are no parameters presenting non-persistent indiffreent cycles. In particular, the Mañé-Sad-Sullivan decomposition trivializes: $\hat{\mathcal{W}}_2$ is a single $J$-stable component. Thus, the straightening map $\chi: \hat{\mathcal{W}}_2 \to \Pol_d$, mapping $w$ to the polynomial $\chi(w)$ which is hybridly conjugate to the polynomial-like map $F_w|_{V_{F_w}}: V_{F_w} \to V$, is analytic (see \cite{douady-hubbard-1985}). Since $\mathbf{f}$ is an analytic family of polynomials, the set
    \[ \mathcal{W} := \{ w \in \hat{\mathcal{W}}_2 \ | \ \chi(w) \text{ is conformally conjugate to some } f_\lambda \} \]
    is an analytic subvariety of $\hat{\mathcal{W}}_2$ --- it can be viewed as the pre-image under $\chi$ of the subset of $\Pol_d$ of classes of polynomials in $\mathbf{f}$, which is a subvariety of $\Pol_d$.
    
\end{proof}

We now have all the tools necessary to prove Theorem \ref{main.analyticity}.

\begin{proof}[Proof of Theorem \ref{main.analyticity}]

    Let $\mathcal{W}$ be the subvariety from Lemma \ref{lema.mating_family_manifold}. Up to a desingularization, we can assume it is in fact a manifold, and thus
    \[ \{ F_w|_{V_{F_w}}: V_{F_w} \to V \}_{w\in \mathcal{W}} \]
    is an analytic family of polynomial-like maps. We wish to show that the set
    \[ \Xi := \{ (\lambda, r, w) \in W'\times \overset{\circ}{\mathcal{D}}_{d+1}\times \mathcal{W} \ | \ F_w \text{ is a mating between } f_\lambda \text{ and } r(H_{d+1}) \} \]
    is an analytic subset of $W'\times \overset{\circ}{\mathcal{D}}_{d+1}\times \mathcal{W}$. From Lemma \ref{lema.group_structure_stability}, along any embedded disk containing $w_0$ we may find a holomorphic motion
    \[ \eta: D\times A_{F_0} \to \overline{\C} \]
    of the annulus $A_{F_0}$ along $D$, based on $F_0$, such that $\eta_w(A_{F_0}) = A_{F_w}$, and $\eta_w$ conjugates the $d:d$ actions of $J\circ F_0|_{A_{F_0}}: A_{F_0} \to A_{F_0}$ and $J\circ F_w|_{A_{F_w}}: A_{F_w} \to A_{F_w}$, for all $w \in D$. Therefore, if $\mu_0$ denotes the trivial Beltrami form, the family of Beltrami forms $\{ \eta_w^\ast\mu_0 \}_{w\in D}$ is an analytic family of Beltrami forms on $A_{F_0}$, all of them invariant under the action of $J\circ F_w|_{A_{F_w}}$. Let $B_0$ be the ''fundamental'' annulus constructed for $r_0$ as in Section 2. From the construction of the mating, we may find a holomorphic map $\phi: A_{F_0} \to B_0$ conjugating the actions of $J\circ F_w|_{A_{F_w}}$ and the correspondence on $B_0$ induced from the powers of $\rho_{r_0}$. Thus, the family $\{ \phi^\ast\eta_w^\ast\mu_0 \}_{w\in D}$ is an analytic family of Beltrami forms on $B_0$, all of them invariant under the action of the powers of $\rho_r$. Thus, there is a map $\hat{r}: D \to \overset{\circ}{\mathcal{D}}_{d+1}$ that to $w \in D$ assigns the representaiton $r$ of $H_{d+1}$ obtained from conjugating $r_0$ with the quasconformal map that integrates $\phi^\ast\eta_w^\ast\mu_0$. Since such integrating maps depend analytically on $w$, we see that $\hat{r}$ must be an analytic map.

    We are essentially done with the proof now. If we denote by $[f] \in \Pol_d$ the conjugacy class of a polynomial $f$, and by $\chi: \mathcal{W} \to \Pol_d$ the straightening map from \cite{douady-hubbard-1985}, associating to each polynomial-like restriction of $F_w$ the class of polynomials hybrid equivalent to it, then we see that $F_w$, $w \in \mathcal{W}$, is a mating between $f_\lambda$ and $r(H_{d+1})$ if an only if $\chi(w) = [f_\lambda]$ and $\hat{r}(w) = r$. Since both $\hat{r}$ and $\chi$ are analytic, as is the map $\lambda \in \Lambda \mapsto [f_\lambda] \in \Pol_d$, we conclude that
    \[ \Xi = \{ (\lambda, r, w) \in W'\times \overset{\circ}{\mathcal{D}}_{d+1}\times \mathcal{W} \ | \ \chi(w) = [f_\lambda] \text{ and } \hat{r}(w) = r \} \]
    is an analytic subset of $W'\times \overset{\circ}{\mathcal{D}}_{d+1}\times \mathcal{W}$. Thus, the mating map is analytic on $W'\times \overset{\circ}{\mathcal{D}}_{d+1}$. Since $W'$ is an open dense subset of $W$, the continuity we got in Theorem \ref{teo.continuity101} implies that the mating map is analytic on $W\times \overset{\circ}{\mathcal{D}}_{d+1}$. Since $W$ was taken an arbitrary connected component of $\overset{\circ}{M}_{\mathbf{f}}$, we conclude that the mating map is analytic on $\overset{\circ}{M}_{\mathbf{f}}\times \overset{\circ}{\mathcal{D}}_{d+1}$.
    
\end{proof}

A simple adaptation of the proof above now gives us Corollary \ref{cor.continuity_along_reps}.

\begin{proof}[Proof of Corollary \ref{cor.continuity_along_reps}]

    Simply consider $\Lambda = \{\lambda_0\}$ a singleton and $f_{\lambda_0} = f$.
    
\end{proof}

\textbf{Remark}: The restrictions to further submanifolds/subvarieties/suborbifolds done in the above Lemmas is necessary since we must localize exactly the matings with members of the family $\mathbf{f}$. Indeed, a map in $\mathcal{R}_{\mathbf{f}}$ might not have a $J$-stable neighborhood between all polynomials --- as is the case for any map of the family $\mathbf{f} = \{f_a(z) = z^3 + az^2 + z\}_{a\in \C}$, which all present a parabolic fixed point at $0$. Still $M_\mathbf{f}$ has non-empty interior: allowing the parabolic fixed point to be persistent produces $J$-stable components within the family.

\subsection{Proof of Theorem \ref{main.continuity}}

We will conclude by looking for the conditions under which the mating map is continuous at points in $\partial M_\mathbf{f}$. As was mentioned in the Introduction, this result will hinge on the quasiconformally rigidity of the parameter.

\begin{deft}

    Let $f$ be a polynomial. We say $f$ is \textit{quasiconformally rigid} if, whenever another polynomial $f'$ is quasiconformally conjugate to $f$, then $f'$ is conformally conjugate to $f$.
    
\end{deft}

\begin{proof}[Proof of Theorem \ref{main.continuity}]

    Let $f = (f_\lambda)_{\lambda\in \Lambda}$ be an analytic family of polynomials. Recall that $\mathcal{T}$ is the set of quasiconformally rigid parameters in $\Lambda$, i.e. the parameters $\lambda$ for which, whenever $f$ is  polynomial quasiconformally conjugate to $f_\lambda$, then it is conformally conjugate. We already have continuity on $\overset{\circ}{M}_f$, so we only need to verify the same for points in $\mathcal{T}$. Fixing $\lambda_0 \in \mathcal{T}$ and $r_0 \in \overset{\circ}{\mathcal{D}}_{d+1}$, we will show that any sequence $(\lambda_n, r_n) \xrightarrow{} (\lambda_0, r_0)$ admits a subsequence $n_k \nearrow \infty$ such that the correspondences $F_{\lambda_{n_k}, r_{n_k}}$ converge to $F_{\lambda_0, r_0}$ locally uniformly. Indeed, we see that the Beltrami coefficients $\mu_{r_n}, \mu_{r_0}$ as defined in the proof of Theorem \ref{teo.continuity101} have uniformly bounded norms in $L^\infty$. Also, since the $h_\lambda$ are obtained from a composition of a holomorphic motion along $\lambda$ and a single quasiconformal map, the $h_{\lambda_n}, h_{\lambda_0}$ have uniformly bounded distortion as well. We thus conclude that the Beltrami coefficients $\mu_{\lambda_n, r_n}, \mu_{\lambda_0, r_0}$ have uniformly bounded norms, and thus their integrating maps have uniformly bounded distortion. From the normality property of quasiconformal maps, we find a subsequence $n_k \nearrow \infty$ such that $\phi_{\lambda_{n_k}, r_{n_k}} \xrightarrow{} \tilde{\phi}$ uniformly for some $\tilde{\phi}$ a quasiconformal map. Since the topological correspondences $G_\lambda$ depend continuously on $\lambda$, we also have that $G_{\lambda_{n_k}} \xrightarrow{} G_{\lambda_0}$ uniformly, and therefore
    \[ F_{\lambda_{n_k}, r_{n_k}} = \phi_{\lambda_{n_k}, r_{n_k}}\circ G_{\lambda_{n_k}}\circ \phi_{\lambda_{n_k}, r_{n_k}}^{-1} \xrightarrow{\quad\quad} \tilde{\phi}\circ G_{\lambda_0}\circ \tilde{\phi}^{-1} =: \tilde{F} \]
    uniformly. Since each of the correspondences $F_{\lambda_{n_k}, r_{n_k}}$ are holomorphic, their limit $\tilde{F}$, which is a branched-covering correspondence because it is conjugate to $G_\lambda$, is also holomorphic. Since it is topologically conjugate to $G_{\lambda_0}$, it also admits a polynomial-like restriction, which is quasiconformally conjugate to $f_{\lambda_0}$. Rigidity thus implies that $\tilde{F}$ must be a mating of $f_{\lambda_0}$ with some representation of $H_{d+1}$ --- indeed, the polynomial-like restriction straightens to some polynomial $f'$, which is quasiconformally conjugate to $f_{\lambda_0}$ (because some polynomial-like restriction of it is), and thus is conformally conjugate to $f_{\lambda_0}$. The fact that the representation has to be $r_0$ comes from the continuity along the $r$ parameter observed after the proof of Theorem \ref{teo.continuity101}.
    
\end{proof}

\textbf{Remark}: The rigidity hypothesis is a natural one. In \cite{douady-hubbard-1985}, Douady and Hubbard show (see Propositions 16) that, when straightening polynomial-like maps (which is also a surgery processes), the resulting polynomial doesn't vary continuously with the parameter in general, for degrees greater than $2$. The key for continuity in degree $2$ is rigidity, and still (see Proposition 15 in the same text) the hybrid conjugation itself fails to vary continuously in general.

As an application, we will now prove Corollaries \ref{cor.unicritical_continuity} and \ref{cor.persistent_parabolic_family}. For that, all that is necessary is to verify that the parameters at the boundaries of the connectedness loci of the families $\{ f_c(z) = z^d + c \}_{c \in \C}$ and $\{ g_a(z) = z^3 + az^2 + z \}_{a\in \C}$ are all quasiconformally rigid. We will actually prove this for a class of one-parameter families.

\begin{deft}

    We say $\mathbf{f} = (f_{\lambda})_{\lambda\in \Lambda}$ is a \textit{topologically complete} family of polynomial maps if, whenever $f$ is a polynomial topologically conjugate to some $f_\lambda$, $\lambda \in \Lambda$, then $f$ is conformally conjugate to some $f_\lambda$, $\lambda \in \Lambda$.
    
\end{deft}

We see that both the previously mentioned families have this property: a polynomial $f$ of degree $d$ is topologically conjugate to some $f_c$ if and only if it is unicritical, and in this case it is conformally conjugate to some $f_c$; and a polynomial $f$ of degree $3$ is topologically conjugate to some $g_a$ if and only if it has a parabolic fixed point of multiplier $1$ (recall that the multiplier of a parabolic fixed point is a topological invariant), and in this case it is conformally conjugate to some $g_a$. The corollaries now follow directly from this proposition (which is just Proposition 7 of \cite{douady-hubbard-1985} repeated for this more general case):

\begin{prop}\label{prop.boundary_rigidity}

    Let $\mathbf{f} = (f_\lambda)_{\lambda \in \Lambda}$ be a topologically complete analytic family of polynomial maps, $\Lambda$ a Riemann surface. Then every parameter in $\partial M_\mathbf{f}$ is quasiconformally rigid.
    
\end{prop}
\begin{proof}

    Let $\lambda \in \partial M_\mathbf{f}$, and $f'$ be a polynomial such that $f_\lambda$ and $f'$ are quasiconformally conjugate under some quasiconformal map $\phi$. If $K_\lambda$ has zero measure, we have nothing to do (the conjugation is hybrid, and the result follows from uniqueness, since the filled Julia sets are connected). In general, we can set $\nu := \overline{\partial}\phi/\partial\phi$ as the Beltrami coefficient of $\phi$, $\nu_0$ as the Beltrami coefficient
    \[ \nu_0(z) := \begin{cases} \nu(z) & \text{ if } z \in K_\lambda; \\ \mu_0(z) = 0  & \text{ otherwise}; \end{cases} \]
    and $\nu_t := t\nu_0$ for any $t \in D(0, 1/\| \nu_0 \|_\infty)$ --- the disk of center $0$ and radius $1/\| \nu_0 \|_\infty$. We can then find an analytic family of integrating maps $\phi_t$ satisfying that $\phi_t(\infty) = \infty$ and $\phi_t(z)/z \xrightarrow{} 1$ as $z \to \infty$. This means that the maps $\phi_t\circ f_\lambda\circ \phi_t^{-1}$ are all polynomials, and furthermore, since they are topologically conjugate to $f_\lambda$, they all belong to the family $\mathbf{f}$ (under further normalization, if necessary). Thus we have a holomorphic map $t \mapsto \lambda(t)$ such that $\phi_t\circ f_\lambda\circ \phi_t^{-1} = f_{\lambda(t)}$. Notice that $\lambda(t) \in M_\mathbf{f}$ for all $t$, but $\lambda(0) = \lambda \in \partial M_\mathbf{f}$ by hypothesis. Thus the association $t \mapsto \lambda(t)$ is constant (since $\Lambda$ is a complex dimension 1 manifold), and in particular $\lambda(1) = \lambda$. Since clearly $f'$ is conformally conjugate to $f_{\lambda(1)}$, we conclude the result.
    
\end{proof}

\bibliographystyle{plain}
\bibliography{biblio}

\end{document}